\documentclass[11pt,reqno,letterpaper]{amsart}

\usepackage{hyperref}
\usepackage{amssymb,amsmath,mathrsfs,amsthm}
\usepackage[dvipsnames]{xcolor}
\usepackage{etoolbox}
\usepackage{fullpage}
\usepackage{calrsfs}
\usepackage[T1]{fontenc} 
\usepackage{amsfonts}
\usepackage{enumitem}
\setenumerate{label={\normalfont(\textit{\roman*})}, ref={\textrm{\roman*}}
}

\patchcmd{\section}{\scshape}{\bfseries}{}{}
\makeatletter
\renewcommand{\@secnumfont}{\bfseries}
\makeatother

\DeclareRobustCommand{\SkipTocEntry}[5]{}

\newtheorem{introtheorem}{Theorem}

\theoremstyle{definition}

\newtheorem*{introdefinition*}{Definition}
\theoremstyle{plain}

\newcommand\Q{{\mathbf Q}}
\newcommand\Z{{\mathbf Z}}
\newcommand\R{{\mathbf R}}
\newcommand\N{{\mathbf N}}

\newcommand\Orb{{\mathcal O}}

\theoremstyle{plain}
\newtheorem{theorem}{Theorem}[section]
\newtheorem{proposition}[theorem]{Proposition}
\newtheorem{lemma}[theorem]{Lemma}
\newtheorem{corollary}[theorem]{Corollary}
\theoremstyle{definition}
\newtheorem{definition}[theorem]{Definition}
\newtheorem{example}[theorem]{Example}
\newtheorem{remark}[theorem]{Remark}
\newtheorem{question}[theorem]{Question}

\renewcommand{\geq}{\geqslant}
\renewcommand{\leq}{\leqslant}

\begin{document}

\date{\today}
 \title{Substitutive systems and a finitary version of Cobham's theorem}

\author[Jakub Byszewski   \and   Jakub Konieczny \and El\.zbieta Krawczyk ]{Jakub Byszewski   \and   Jakub Konieczny \and El\.zbieta Krawczyk}

\address[JB]{Faculty of Mathematics and Computer Science\\Institute of Mathematics\\
Jagiellonian University\\
Stanis\l{}awa \L{}ojasiewicza 6\\
30-348 Krak\'{o}w}
\email{jakub.byszewski@gmail.com}

\address[JK]{Einstein Institute of Mathematics\\ Edmond J. Safra Campus\\ The Hebrew University of Jerusalem\\ Givat Ram\\ Jerusalem, 9190401\\ Israel}
\email{jakub.konieczny@gmail.com}

\address[EK]{Faculty of Mathematics and Computer Science\\Institute of Mathematics\\
Jagiellonian University\\
Stanis\l{}awa \L{}ojasiewicza 6\\
30-348 Krak\'{o}w}
\email{ela.krawczyk7@gmail.com}


\maketitle
\begin{abstract} We study substitutive systems generated by nonprimitive substitutions and show that transitive subsystems of substitutive systems are substitutive. As an application we obtain  a  complete characterisation of the sets of words that can appear as common factors of two automatic sequences defined over multiplicatively independent bases. This generalises the famous theorem of Cobham. 
 \end{abstract}

\section*{Introduction}

Let $\mathcal{A}$ be a finite alphabet, let $\mathcal{A}^*$ be the set of finite words over $\mathcal{A}$ and let $\mathcal{A}^{\omega}$ be the set of sequences $(a_n)_{n\geq 0}$ with values in $\mathcal{A}$. A sequence in $\mathcal{A}^{\omega}$ is called \textit{purely substitutive} if it is a fixed point of some substitution $\varphi\colon \mathcal{A}\rightarrow \mathcal{A}^*$, meaning that the sequence does not change after replacing each letter $a\in\mathcal{A}$ by a finite word $\varphi(a)$. A sequence in $\mathcal{A}^{\omega}$ is called \textit{substitutive} if it arises from a purely substitutive sequence over some alphabet $\mathcal{B}$ after applying a (possibly non-injective) map $\pi\colon \mathcal{B}\rightarrow \mathcal{A}$. We say that a dynamical system $X\subseteq \mathcal{A}^{\omega}$ is  \emph{substitutive}  if it arises as the orbit closure of a  substitutive sequence. Such systems were extensively studied in the context of primitive substitutions \cite{Lothaire-book1,Lothaire-book2,Queffelec-book}, necessarily restricting such studies to minimal systems. There is also a close relationship between substitutive systems and D0L-systems (see, e.g.\ \cite{RS-1}, and compare with \cite{KloudaStarosta-2015}, where the authors solve a problem on D0L-systems that is related to the problems considered below). In the recent years there has been growing interest in the study of nonminimal substitutive systems, e.g.\ with connection to Bratteli diagrams \cite{BKM} and tiling spaces \cite{MR}. 

Nevertheless, it seems that  treatments of substitutive systems arising from nonprimitive substitutions are still scarce. In particular, the following basic question seems not to have been studied: Is every transitive subsystem of a substitutive system substitutive? In other words, if $X$ is a substitutive system and $x$ is a sequence in $X$, is there  a substitutive sequence $y$ such that $x$ and $y$ have the same set of factors? The same question can be posed for $k$-automatic systems (for the precise definitions of substitutive and $k$-automatic systems see Section 1). Note that  substitutive systems may contain uncountably many points, while the number of substitutive sequences is countable, and so most sequences in a substitutive system will often not be substitutive. 

The aim of this paper is twofold. First, we study general substitutive systems and provide a positive answer to the above question. Second, we apply this result to obtain a finitary version of the classical theorem of Cobham, answering a question posed by Shallit\footnote{Jeffrey Shallit, journal entry, 2 February 2018 (private communication)}  (see also the discussion in \cite{MRSS}).

We focus our study on noninvertible substitutive systems, but we briefly present  analogous results for invertible systems as well.  
Noninvertible substitutive systems have a considerably more complicated and interesting dynamical structure than the invertible ones. For instance, it follows from \cite{MR} that the number of subsystems of an invertible substitutive system is finite (see Remark \ref{remark:twosided}), while noninvertible substitutive systems can have infinitely many subsystems (see e.g.\ Example  \ref{infinite}).

Throughout the article, we only consider substitutions that are growing (i.e.\ substitutions $\varphi\colon \mathcal{A} \to \mathcal{A}^*$ such that the  length of the words $\varphi^n(a)$ tends to infinity for all letters $a\in \mathcal{A}$). We do not know what happens when this assumption is removed, though we would not be surprised if it could be shown that Theorem \ref{mainthmA} below continues to hold.

Our first main result yields the following description of transitive subsystems of substitutive systems.

\begin{introtheorem}\label{mainthmA} Every transitive subsystem of a substitutive system is substitutive. Every transitive subsystem of a $k$-automatic system is $k$-automatic.
\end{introtheorem}

In fact, we obtain a much more precise description of  substitutive (resp., $k$-automatic) sequences generating such subsystems. We present here a simplified version of the result in the noninvertible case.

\begin{introtheorem}\label{mainthmB} Let $x$ be a purely substitutive sequence produced by a substitution $\varphi\colon \mathcal{A}\rightarrow\mathcal{A}^*$, and let $X$ be the orbit closure of $x$. There exists a power $\tau=\varphi^m$ of $\varphi$ and a finite set of words $W\subset \mathcal{A}^*$ such that every transitive subsystem $Y\subset X$ can be generated by a sequence $y\in X$ that is a suffix of a biinfinite sequence of the form
\begin{equation}\label{sequence}
\cdots \tau^2(v)\tau(v)vabw\tau(w) \tau^2(w) \cdots
\end{equation}
for some $v\in W$, $w\in W \setminus \{\epsilon\}$, and $a,b\in \mathcal{A}\cup \{\epsilon\}$.
\end{introtheorem}

Substitutive sequences of such a form have been considered before in specific contexts. Let $x$ be a substitutive sequence over an alphabet $\mathcal{A}$. A sequence $z$ in the orbit closure of $x$ is called extremal if it is lexicographically minimal with respect to some total order on the alphabet $\mathcal{A}$. In \cite{CRSZ} it was shown that (under some additional assumptions) all extremal sequences are substitutive. We note here a curious observation that all extremal sequences in (purely) substitutive systems considered in \cite{CRSZ} are of the form \eqref{sequence} (\cite[Lemma 9]{CRSZ}).

In the second part of the article we restrict our attention to automatic sequences. One of the most fundamental results about automatic sequences  is Cobham's theorem, which gives a strong relation between $k$-automaticity of a sequence  and the chosen base $k$. Recall that two integers $k,l\geq 2$ are called \textit{multiplicatively independent} if they are not both powers of the same integer (equivalently, if they are not rational powers of each other).  Cobham's theorem states that a sequence is simultaneously automatic with respect to two multiplicatively independent bases if and only if it is ultimately periodic \cite{Cobham-1969}. This result has sparked a lot of research and has been generalised to a variety of different settings.  An extension of Cobham's theorem to the class of substitutive sequences was obtained by Durand in 2011 \cite{Durand-2011}.

A considerable effort went also into strengthening Cobham's original  theorem. Let $x$ and $y$ be two automatic sequences defined over multiplicatively independent bases. In \cite{Fagnot-1997}, Fagnot showed that for the claim of Cobham's theorem to hold for $x$ and $y$ it is sufficient that they contain the same factors; that is, if the languages $\mathcal{L}(x)$ and $\mathcal{L}(y)$ coincide, then both $x$ and $y$ are ultimately periodic. In \cite{BK} the first- and second-named authors showed that the claim of Cobham's theorem holds if the sequences $x$ and $y$ agree on a set of upper density 1.  In the spirit of Shallit's question, Mol, Rampersad, Shallit and Stipulanti obtained in \cite{MRSS} an explicit bound on the length of a common prefix of $x$ and $y$ that depends on the number of states in the automata generating $x$ and $y$. They further asked for a characterisation of the set $\mathcal{L}(x)\cap \mathcal{L}(y)$  of common factors of  $x$ and $y$. Since all ultimately periodic sequences are $k$-automatic for all $k\geq 2$, it is clear that we cannot hope for a bound on the length of common factors of $x$ and $y$. We might hope, however, that the set of common factors exhibits some simple periodic-like structure. 

In this paper, we show the following finitary version of Cobham's theorem, which provides a  complete characterisation of the sets of words that can appear as common factors of two automatic sequences defined over multiplicatively independent bases. In particular, this set can always be described by a finite amount of data.

\begin{introtheorem}\label{mainthmC} Let $k,l\geq 2$ be multiplicatively independent integers, let $\mathcal{A}$ be an alphabet, and let $U\subset \mathcal{A}^*$. The following conditions are equivalent:
\begin{enumerate}
\item\label{thm:maincobhami-0} there exist a $k$-automatic sequence $x$ and an $l$-automatic sequence $y$ such that $U$ is  the set of common factors of $x$ and $y$;
\item\label{thm:maincobhamii-0}  the set $U$ is a finite nonempty union of sets of the form  $\mathcal{L}(^{\omega}vuw^{\omega})$, where $u,v,w$ are (possibly empty) words over $\mathcal{A}$ and $^{\omega}vuw^{\omega}=\cdots v v v u w w w \cdots.$
\end{enumerate}
\end{introtheorem}

Note that Cobham's theorem follows immediately from Theorem \ref{mainthmC}. One of the crucial ingredients in the proof of Theorem \ref{mainthmC} is Theorem \ref{mainthmA} applied to $k$-automatic systems.  Indeed, as a simple application of Theorem \ref{mainthmA} and Fagnot's result we can already obtain the following generalisation: if $\mathcal{L}(z)\subset \mathcal{L}(x)\cap\mathcal{L}(y)$ for some $k$-automatic sequence $x$ and $l$-automatic sequence $y$, then $z$ is ultimately periodic (see Corollary \ref{cor:cobpoints}).

Cobham's original result can be understood in the framework of recognisability of subsets of integers in base-$k$ numeration systems. In this context, Cobham's theorem has been seen to hold in many nonstandard numeration systems over the integers as well as in the higher dimensional setting over $\mathbf{N}^d$, the latter due to Semenov  \cite{Semenov-1977}.  For a comprehensive treatment of these developments  and the interplay between substitutions, numeration systems and  logic, see the surveys \cite{BHMV} and \cite{DR-2011}.

In the setting of recognisable subsets of $\mathbf{R}^d$, an analogue of Cobham's theorem for integer-based numeration systems was obtained by Boigelot, Brusten and Leroux \cite{BB-2009, BBL-2009} with recognisability being defined with respect to (weak) B{\"u}chi automata (for a weaker result for general B{\"u}chi automata, see \cite{BBB-coprime}).  The one-dimensional case was obtained independently by Adamczewski and Bell \cite{AB-2011}, although it was framed in a different language inspired by the kernel-based definition of automaticity. The two approaches were thoroughly linked in \cite{CLR-2015}, providing further connections with (graph directed) iterated function systems and Cobham-like theorems for iterated function systems obtained in \cite{FW-2009} and \cite{EKM-2010}. For more about these developments we refer to \cite{CLR-2015} and references therein.   

In another direction, Cobham's theorem proved amenable to various algebraic extensions, in part thanks to the characterisation of automaticity in terms of algebraicity of power series over $\mathbf{F}_p$ obtained by Christol \cite{Christol}. Cobham's theorem has been successfully generalised to the class of regular  sequences \cite{Bell-2005}, quasi-automatic functions \cite{AB-2008} (introduced by Kedlaya in \cite{Kedlaya} in order to give a description of the algebraic closure of the field  $\mathbf{F}_q(t)$) and Mahler functions (over fields of characteristic 0) \cite{AB-2013}. All these proofs made use of Cobham's original theorem. A much simpler proof of Cobham's theorem for Mahler functions that does not rely on the original result of Cobham has been obtained by Sch\"afke and Singer \cite{SS-2017}.

It would be interesting to see if the finitary version of Cobham's theorem can be generalised to any of these wider settings. It is easy to see that we cannot hope for a straightforward extension to the class of substitutive sequences since we can construct two non-ultimately periodic substitutive sequences $x$ and $y$ over multiplicatively independent bases such that $\mathcal{L}(x) \subset \mathcal{L}(y)$ (see Remark \ref{remark:substitutive} for more details).
 It is, however, reasonable to hope that some extension of Theorem \ref{mainthmC} holds in a higher dimensional setting and can lead to a generalisation of the Cobham--Semenov theorem. A possible approach could involve extending Theorem \ref{mainthmA} to automatic systems over $\mathbf{N}^d$.

We briefly discuss the contents of the paper. In the first section, we recap some basic facts about (topological) dynamical systems and substitutions, and introduce the class of substitutive systems that we will work with. In Definition \ref{def:idemsub} we introduce the notion of an idempotent substitution and show that for every substitution $\varphi$ some power $\varphi^n$ is idempotent. Idempotency gathers all the technical assumptions that we need from the substitution in order to carry out the proofs in Section 2. In Lemma \ref{lem:morvarphicseq} we  show that certain sequences that will turn out to be closely linked with sequences generating  transitive subsystems of substitutive (resp., $k$-automatic) systems are indeed  substitutive (resp., $k$-automatic).

The second section is devoted to the classification of minimal and transitive subsystems of substitutive systems, and contains the proofs of Theorems \ref{mainthmA} and \ref{mainthmB}. The analogues of Theorems \ref{mainthmA} and \ref{mainthmB} for invertible substitutive systems are presented at the end of the section.

The proof of Theorem \ref{mainthmC} is given in the third section. The main ingredients in the proof are Theorem \ref{mainthmA} and Theorem \ref{prop:recurrence}, which describes occurrences of cyclic factors in automatic sequences. We also discuss the problem of effective computability of the set of common factors, that is, existence of an algorithm that, given as  input two automatic sequences $x,y\in \mathcal{A}^{\omega}$ defined over multiplicatively independent bases, returns a finite set of triples of words $(v,u,w)\in (\mathcal{A}^*)^3$ that describe the set of common factors $\mathcal{L}(x)\cap \mathcal{L}(y)$ in the sense of Theorem \ref{mainthmC}. Our proof of Theorem \ref{mainthmC} uses  the compactness of the space $\mathcal{A}^{\omega}$  and is not effective. We believe that the question of whether Theorem \ref{mainthmC} admits an effective proof is interesting and worthy of further study.

\section*{Acknowledgements}
This research was supported by National Science Centre, Poland grant
number 2018/29/B/ST1/01340 (Jakub Byszewski), ERC grant ErgComNum 682150 (Jakub Konieczny) and National Science Centre, Poland grant
number 2012/07/E/ST1/00185 (El\.zbieta Krawczyk).
The authors are very grateful to Jeffrey Shallit, who suggested this line of research, as well as to Dominik Kwietniak, Clemens M\"ullner and Tamar Ziegler. We also wish to thank the anonymous referees for their detailed comments and suggestions, and in particular for pointing out the reference \cite{AS} and suggesting the streamlined proof of Lemma \ref{lem:morvarphicseq}\eqref{lem:morvarphicseq2}.

\section{Preliminaries}\label{sec:preliminaries}
	
In this section we recall some classical definitions and state a few preliminary lemmas.

\subsection*{Symbolic dynamics}

\textit{A (topological) dynamical system} is a compact metric space $X$ together with a continuous map $T\colon X\rightarrow X$. Let $T^n$ denote the $n$-th iterate of $T$, and let $\Orb(x)=\{T^n(x)\mid n\geq 0\}$ denote the \emph{orbit} of a point $x\in X$. A point $x\in X$ is \textit{periodic} if $T^k(x)=x$ for some $k\geq 1$. A point $x\in X$ is \textit{ultimately periodic} if there exists $m\geq 0$ such that $T^m(x)$ is periodic. A \textit{subsystem} of $X$ is a closed subset of $X$ that is invariant under the map $T$, i.e.\ a closed set $X' \subset X$ with $T(X')\subset X'$. A system $X$ is called \textit{minimal} if $X \neq \emptyset$ and if $X$ has no subsystems other than $\emptyset$ and $X$; equivalently, a system $X\neq \emptyset$ is minimal if the orbit of every point is dense in $X$ \cite[Ex.\ 4.2.1.a]{EinsiedlerWard-book}. A system $X$ is called \textit{transitive} if it has a point with a dense orbit. An easy application of Zorn's lemma shows that every dynamical system has a  minimal subsystem  \cite[Ex.\ 4.2.1.c]{EinsiedlerWard-book}. We say that a dynamical system $(Y,S)$ is a \textit{(topological) factor} of the system $(X,T)$ if there exists a continuous surjective map $\pi\colon X\rightarrow Y$  such that $\pi\circ T=S\circ \pi$. Such a map $\pi$ is called \textit{a factor map}. We will need the following simple fact.

\begin{lemma}\label{lem:factor}
Let $X$ and $Y$ be dynamical systems and let $\pi \colon X \to Y$ be a factor map. Let $Y'$ be a minimal subsystem of $Y$. Then there exists a minimal subsystem $X'$ of $X$ such that $\pi(X')=Y'$.
\end{lemma}
\begin{proof}
Let $X''\subset X$ be the preimage of $Y'$ by the map $\pi$. Clearly, $X''$ is a subsystem of $X$. Let $X'$ be some minimal subsystem of $X''$. Then $\pi(X')$ is a subsystem of $Y'$, and since $Y'$ is minimal, $\pi(X')=Y'$.
\end{proof}

In this paper we are interested in dynamical systems coming from substitutive sequences. Let $\mathcal{A}$ be a finite set (called an \emph{alphabet}).  Let $\mathcal{A}^*$ denote the set of finite words over $\mathcal{A}$. This is a monoid under concatenation. The empty word is denoted by $\epsilon$. We say that a word $w$ is a \emph{factor} of a word $v$ or that $w$ \textit{appears} in $v$ if $v=ywz$ for some words $y$ and $z$. A word $w$ is a \textit{prefix} of a word $v$  if $v = wz$ for some word $z$. A prefix $w$ is \textit{proper} if $w\neq v$. We similarly define a \textit{suffix}. For a word $w$ we let $|w|$ denote the length of $w$.

Let $\mathcal{A}^{\omega}$ denote the set of sequences over $\mathcal{A}$.   For a sequence $x$   and integers $i\leq j$ we write  $x_{[i,\,j-1]}$ for the word $x_ix_{i+1}\cdots x_{j-1}$ and $x_{[i,\,\infty)}$ for the sequence $x_ix_{i+1}\cdots$. (In particular, $x_{[i,\,i-1]}=\epsilon$.) The notions of concatenation, factor, prefix and suffix are used for words and sequences as long as they make obvious sense. For a word $w\neq \epsilon$ let $w^{\omega}$ denote the sequence $w^{\omega}=w w w \cdots$, and we put $\epsilon^{\omega}=\epsilon$. While we will always use the notation $x_n$ for the terms of a sequence $x=(x_n)_{n\geq 0}\in\mathcal{A}^{\omega}$, we regard words themselves as not indexed.

The set $\mathcal{A}^{\omega}$ with the product topology (where we use discrete topology on each copy of $\mathcal{A}$) is a compact metrisable space. We define the shift map $T\colon \mathcal{A}^{\omega}\rightarrow \mathcal{A}^{\omega}$ by $T((x_n)_n)= (x_{n+1})_n$. The space $\mathcal{A}^{\omega}$ together with the shift map $T$ is a dynamical system. We refer to subsystems $X$ of $\mathcal{A}^{\omega}$ as \textit{subshifts}. Let $\mathcal{L}(X)$ denote the \textit{language} of the subshift $X$, i.e.\ the set of all finite words that appear in some $x\in X$. A subshift $X$  is uniquely determined by its language since \[X=\{x\in \mathcal{A}^{\omega} \mid \text{ all factors of } x \text{ are in } \mathcal{L}(X)\}.\] We also use $\mathcal{L}(y)$ to denote the set of factors of a word or a sequence $y$.  By a slight abuse of terminology we say that a sequence of words $w_n$ converges to the sequence $x$ if $|w_n|\to \infty$ and for every $m\geq 0$ the prefixes of $x$ and $w_n$ of length $m$ agree for sufficiently large $n$.

We will occasionally also work with backwards infinite sequences in ${}^{\omega}\!\mathcal{A}$ and biinfinite sequences  in ${}^{\omega}\!\mathcal{A}^{\omega}$. The definitions of factor, prefix, suffix and language  generalise to these cases in a straightforward manner. We always regard sequences $(a_n)_n$ in $\mathcal{A}^{\omega}$ as indexed by $n\in\{0,1,\ldots\}$, backwards infinite sequences in ${}^{\omega}\!\mathcal{A}$ as indexed by $n\in\{\ldots,-2,-1\}$, and biinfinite sequences in ${}^{\omega}\!\mathcal{A}^{\omega}$ as indexed by $n\in \{\ldots,-2,-1,0,1,\ldots\}$.

\subsection*{Substitutive sequences}

 Let $\mathcal{A}$ be an alphabet. A \textit{substitution} is a map $\varphi\colon \mathcal{A}\rightarrow \mathcal{A}^*$ that assigns to each letter $a$ some finite word $w$ in $\mathcal{A}^*$. We only consider substitutions that are \textit{growing}, i.e.\  $|\varphi^n(a)|\rightarrow \infty$ as $n\rightarrow \infty$ for each $a\in\mathcal{A}$, and throughout the paper the term `substitution' is used for a growing substitution. A substitution $\varphi$ is called \textit{primitive} if there exists an integer $n \geq 1$ such that for any $a,b\in \mathcal{A}$ the letter $a$ appears in $\varphi^n(b)$. A letter $a\in \mathcal{A}$ is \textit{prolongable} if $a$ is the initial letter of $\varphi(a)$. A letter $a\in \mathcal{A}$ is \textit{backwards prolongable} if $a$ is the final letter of $\varphi(a)$. If $a$ is prolongable (resp., backwards prolongable), then the sequence $\varphi^n(a)$ converges to a sequence in $\mathcal{A}^{\omega}$ (resp., in ${}^{\omega}\!\mathcal{A})$, denoted by $\varphi^{\omega}(a)$ (resp., ${}^{\omega}\!\varphi(a)$). A \textit{coding} is an arbitrary map $\pi\colon \mathcal{A}\rightarrow \mathcal{B}$ between alphabets $\mathcal{A}$ and $\mathcal{B}$. A surjective coding $\pi$ naturally extends to a factor map $\pi\colon\mathcal{A}^{\omega}\rightarrow \mathcal{B}^{\omega}$  between dynamical systems.

A substitution $\varphi\colon \mathcal{A}\rightarrow \mathcal{A}^*$ induces a natural map $\varphi\colon \mathcal{A}^{\omega} \to \mathcal{A}^{\omega}$, denoted by the same letter. We say that a sequence $x$ is \textit{purely substitutive}  if it is a fixed point of some substitution $\varphi$, i.e.\ $\varphi(x)=x$. In this case we  also say that the sequence $x$ is \emph{produced} by the substitution $\varphi$. Sequences produced by a substitution $\varphi$ are exactly of the form $\varphi^{\omega}(a)$ for a prolongable letter $a$. A \textit{substitutive} sequence is the image of a purely substitutive sequence under a coding. 

We say that a substitution $\varphi\colon \mathcal{A}\rightarrow \mathcal{A}^*$ is \textit{of constant length $k$} if $|\varphi(a)|=k$ for each $a\in \mathcal{A}$. A fixed point of a substitution of constant length $k$ is called a \textit{purely}  $k$-\textit{automatic sequence}. A $k$-\textit{automatic} sequence is the image of a purely $k$-automatic sequence under a coding. The classes of substitutive and $k$-automatic sequences are invariant under changing finitely many terms of a sequence and under the forward and backward shift operations  \cite[Cor.\ 6.8.5 and Thm.\ 7.6.1 \& 7.6.3]{AlloucheShallit-book}. We also mention here the trivial case of Cobham's theorem, which says that for any integer $t\geq 1$ the classes of $k$-automatic and $k^t$-automatic sequences coincide \cite[Theorem 6.6.3]{AlloucheShallit-book}.

The term \emph{automatic} has its origin in theoretical computer science. Informally speaking, automata (or more precisely finite deterministic $k$-automata with output) are simple finite computational devices that compute the $n$-th term of a sequence from the base-$k$ digits of $n$. For more details, see \cite{AlloucheShallit-book}. A result of Cobham asserts that the description of $k$-automatic sequences in terms of substitutions of constant length $k$ and in terms of $k$-automata are equivalent \cite[Thm.\ 6.3.2]{AlloucheShallit-book}, \cite{Cobham-1972}. In this paper, we will only work with the former definition.
Note that the assumption that a substitution is growing is trivially satisfied when the substitution is of constant length $k\geq 2$.

Yet another definition of an automatic sequence can be given in terms of kernels. Let $x=(x_n)_{n\geq 0}$ be a sequence over an alphabet $\mathcal{A}$.  Let $\N=\{0,1,\dots\}$ denote the nonnegative integers, and let $k\geq 2$ be an integer. The $k$-\textit{kernel} of $x$ is defined as the family of sequences\[\mathrm{K}_k(x)=\{(x_{k^m n +r})_{n\geq 0}\mid m,r \in \N, 0\leq r< k^m\}.\]A theorem of Cobham asserts that a sequence is $k$-automatic if and only if its $k$-kernel is finite (see \cite{Cobham-1972} or \cite[Thm.\ 6.6.2]{AlloucheShallit-book}). 

Later we will need the following result. 

\begin{lemma} \label{lem:morvarphicseq} Let $\varphi\colon \mathcal{A}\rightarrow \mathcal{A}^{*}$ be a substitution and let $w\in \mathcal{A}^*$ be nonempty. Consider the sequence $x=w\varphi(w)\varphi^2(w)\cdots$. \begin{enumerate} 
\item \label{lem:morvarphicseq1}  The sequence $x$ is substitutive. 
\item \label{lem:morvarphicseq2} If $\varphi$ is of constant length $k$, then $x$ is $k$-automatic.
\end{enumerate}
\end{lemma}

\begin{proof} Part \eqref{lem:morvarphicseq1}  follows from \cite[Lemma 5]{CRSZ}; since the argument is short, we include it for completeness. Let $\spadesuit$ be a letter not belonging to the alphabet $\mathcal{A}$ and consider the  substitution $\tau\colon \mathcal{A}\cup\{\spadesuit\}\rightarrow (\mathcal{A}\cup\{\spadesuit\})^*$ given by $\tau(\spadesuit)=\spadesuit w$, $\tau(a)=\varphi(a)$ for $a\in \mathcal{A}$. The sequence $\tau^{\omega}(\spadesuit)$ takes the form $\spadesuit x$ and is clearly substitutive. Therefore, since the class of substitutive sequences is closed under shifts, $x$ is substitutive as well. 

We now assume that $\varphi$ is of constant length $k$ and prove part \eqref{lem:morvarphicseq2}. Intuitively, the claim holds since the equation $x=w\varphi(x)$ allows us to express the elements of the kernel of $x$ in terms of codings of finitely many shifts of the sequence $x$. For a formal proof, put $s=|w|$.  For an integer $j\in [0,k-1]$ write $j-s=kc_j+r_j$ for integers $c_j, r_j$ with $r_j\in[0,k-1]$. Since $x=w\varphi(x)$, we see that $$x_{kn+j}=(\varphi(x))_{kn+j-s}=(\varphi(x))_{k(n+c_{j})+r_{j}}$$ whenever $kn+j\geq s$. Hence, for all sufficiently large $n$, $x_{kn+j}$ is the $(r_j+1)$-st letter of the word $\varphi(x_{n+c_{j}})$. The claim then follows from \cite[Thm. 2.2]{AS}.
\end{proof}

\subsection*{Substitutive systems}

 Let $\mathcal{A}$ be an alphabet. A system $X\subseteq \mathcal{A}^{\omega}$ is called \emph{purely substitutive} (resp., \emph{substitutive}, $k$-\emph{automatic}) if it arises as the orbit closure of a purely substitutive (resp., substitutive, $k$-automatic) sequence. Note that any such system is automatically transitive. 

There is a more general notion of systems arising from substitutions.  Let $\varphi\colon \mathcal{A}\rightarrow \mathcal{A}^*$ be a substitution and let $X_{\varphi}$ denote the dynamical system generated by  $\varphi$, i.e.\
\[
X_{\varphi}=\{z\in \mathcal{A}^{\omega} \mid \text{ every factor of } z  \text{ appears in }\varphi^n(a) \text{ for some } n\geq 0 \textrm{ and } a\in\mathcal{A}\}.
\]The system $X_{\varphi}$ does not have to be transitive; consider, e.g.\ the substitution $\varphi\colon \{0,1\}\rightarrow \{0,1\}^*$ given by $\varphi(0)=00$, $\varphi(1)=11$ for which $X_{\varphi}=\{0^{\omega},1^{\omega}\}$. It is clear that every substitutive system has the form $\pi(X_{\varphi})$ for some substitution $\varphi\colon \mathcal{A}\rightarrow \mathcal{A}^*$ and coding $\pi\colon\mathcal{A}\rightarrow\mathcal{B}$. It is also well-known that if $X_{\varphi}$ is minimal, then it is substitutive \cite[Section 5.2]{Queffelec-book}. Proposition \ref{prop:subsystemsXb} below shows more generally that  a system $X_{\varphi}$ is substitutive if and only if it is transitive.

Let $\varphi\colon \mathcal{A}\rightarrow \mathcal{A}^*$ be a substitution. For $b\in \mathcal{A}$ let $\mathcal{A}_b$ denote the set of all letters $a\in\mathcal{A}$ appearing in $\varphi^n(b)$ for some $n\geq 0$. The substitution $\varphi$ maps $\mathcal{A}_b$ to $\mathcal{A}_b^*$. For $a,b\in \mathcal{A}$, we write $b\succcurlyeq a$ if $a\in \mathcal{A}_b$. Note that the relation $\succcurlyeq $ is only a preorder on the set $\mathcal{A}$. If $b\succcurlyeq  a \succcurlyeq  b$, we write $a\sim b$. We also write $b\succ a$ if $b\succcurlyeq a$ and $b\not\sim a$. The relation  $\sim$ is an equivalence relation and $\succ$ is a strict partial order  on $\mathcal{A}$. These relations clearly depend on the substitution $\varphi$, but we will  nevertheless write $ b\succ a$ or $b\succcurlyeq a$ and call a letter minimal, maximal or equivalent to another letter when the substitution is clear from the context. For $b\in \mathcal{A}$  let $X_{\varphi, b}$ denote the subsystem of $X_{\varphi}$ generated by $b$, i.e.\
\[
X_{\varphi, b}=\{ z\in X_{\varphi}\mid \text{ every factor of } z  \text{ appears in } \varphi^n(b) \text{ for some } n\geq 0\}.
\] 	 	
Thus, $X_{\varphi, b}$ is equal to the system $X_{\varphi'}$, where $\varphi'$ is the substitution $\varphi$ restricted to the alphabet $\mathcal{A}_b$. We will often write $X_b$ instead of $X_{\varphi,b}$ when the substitution $\varphi$ is clear from the context.  Note that if $a\succcurlyeq  b$, then $X_a\supseteq X_b$. Note also that $X_{\varphi}=\bigcup_{a\in\mathcal{A}} X_a$ and  $\varphi(X_b)\subset X_b$ for each $b\in\mathcal{A}$. 

The following lemma is a variant of \cite[Lemma 6]{CRSZ} (cf.\ \cite[Prop. 5.10]{BKM} for a version for two-sided dynamical systems). 

\begin{lemma}\label{lem:prefix}
Let $\varphi\colon \mathcal{A}\rightarrow \mathcal{A}^*$ be a substitution and let $x\in X_{\varphi}$. There exist $y\in X_{\varphi}$ and a proper prefix $u$ of $\varphi(y_0)$ (possibly empty) such that $ux=\varphi(y)$.
\end{lemma}
\begin{proof}
Let $v_n$ be the prefix of $x$ of length $n\geq 1$. Since $x\in X_{\varphi}$, there exist $b\in\mathcal{A}$ and $k\geq 1$ such that $v_n$ is a factor of $\varphi^k(b)$. Removing from $\varphi^{k-1}(b)$ the longest prefix whose image by $\varphi$ does not intersect $v_n$, we obtain a suffix $w_n$ of $\varphi^{k-1}(b)$ with initial letter $a_n$ and a proper prefix $u_n$ of $\varphi(a_n)$ such that $u_nx$ and $\varphi(w_n)$ agree on the first $n$ positions. 

Since there are only finitely many possibilities for $u_n$ and $\mathcal{A}^{\omega}$ is compact, there exist $u\in \mathcal{A}^*$, $y\in X_{\varphi}$ and an increasing sequence $(k_m)_{m\geq 0}$ of positive integers such that $u_{k_m}=u$ for all $m \geq 0$ and $w_{k_m}$ converges to $y$ as $m \to \infty$. By construction, $u$ is then a proper prefix of $\varphi(y_0)$ and  we have $ux=\varphi(y)$.
\end{proof}

The following lemma is well-known for primitive substitutions (see, e.g.\ \cite[Prop.\ 5.4]{Queffelec-book}); we will however need the claim under the weaker assumption of transitivity. 

\begin{lemma}\label{lem:nthpower}
Let $\varphi\colon \mathcal{A}\rightarrow \mathcal{A}^*$ be a substitution. If $X_{\varphi}$ is transitive, then $X_{\varphi}=X_{\varphi^n}$ for any $n\geq 1$.
\end{lemma}
\begin{proof}
Let $n\geq 1$ and for an integer $i\in [0,n-1]$ consider the subsystems $X_i$ of $X_{\varphi}$ defined as
\[
X_i=\{x\in X_{\varphi} \mid \textrm{ every factor of } x \textrm{ appears in } \varphi^k(b) \textrm{ for some } b\in \mathcal{A} \textrm{ and } k\equiv i \bmod n\}.
\] Note that $X_0=X_{\varphi^n}$ and $X_{\varphi}=\bigcup_{i=0}^{n-1} X_i$. Since $X_{\varphi}$ is transitive, $X_{\varphi}=X_j$ for some integer $j\in [0,n-1]$. Let $x\in X_{\varphi}$. By Lemma \ref{lem:prefix} there exist $y\in X_{\varphi}=X_j$ and a proper prefix $u$ of $\varphi(y_0)$ such that $ux=\varphi(y)$. Since every factor of $y$ is a factor of $\varphi^k(b)$ for some $b\in\mathcal{A}$ and $k\equiv j \pmod n$, every factor of $x$ is a factor of $\varphi^{k+1}(b)$ and hence $x$ lies in $X_{j+1}$ (where $X_n=X_0$). Thus $X_{\varphi}=X_{j+1}$. Repeating the argument, we get that $X_{\varphi}=X_i$ for all integers $i\in [0,n-1]$. In particular, $X_{\varphi}=X_0=X_{\varphi^n}$.
\end{proof}
\begin{remark} The result above is not necessarily true without the assumption of transitivity. For an example, consider the substitution $\varphi\colon\{0,1,2\} \to \{0,1,2\}^*$ given by $\varphi(0)=12$, $\varphi(1)=22$, $\varphi(2)=11$. Then \[X_{\varphi}=\{1^n 2^{\omega} \mid n\geq 0\} \cup \{2^n 1^{\omega} \mid n\geq 0\} \quad \text{and} \quad X_{\varphi^2}= \{2^n 1^{\omega} \mid n\geq 0\}\cup \{2^{\omega}\}.\]
\end{remark}

Let $\varphi\colon \mathcal{A}\rightarrow \mathcal{A}^*$ be a substitution. We  say that a letter $a\in \mathcal{A}$ is \textit{ample}  if $a$ appears  in $\varphi^n(a)$ for some $n\geq 1$, and \textit{very ample}  if $a$ appears at least twice in $\varphi^n(a)$ for some $n\geq 1$. Note that for all $n\geq 1$ the sets of ample and very ample letters with respect to substitutions $\varphi$ and $\varphi^n$ are the same. The set of all ample letters is denoted by $\mathcal{A}'$.

We note two easy properties of ampleness. First, a letter equivalent to an ample letter is itself ample. Second, for any ample letter $a$ the word $\varphi(a)$ contains at least one letter equivalent to $a$. For $a\in \mathcal{A}'$ let $\lambda_{\varphi}(a)$ denote the letter $b$ which is equivalent to $a$ and which occurs in $\varphi(a)$ at the last position among all letters equivalent to $a$. This gives rise to a map $\lambda_{\varphi}\colon \mathcal{A}'\rightarrow \mathcal{A}'$. Finally, we claim that $\lambda_{\varphi}^n=\lambda_{\varphi^n}$ for all $n\geq 1$. In fact, for any letter $a$, we may write $\varphi(a)$ in the form $\varphi(a)=u\lambda_{\varphi}(a)v$ for words $u$ and $v$ such that $v$ contains only letters $b \prec a$. For any such letter $b$, the word $\varphi^{n-1}(b)$ also contains only letters $c \prec b$, and hence none of these letters is equivalent to $a$ (or equivalently $\lambda_{\varphi}(a)$). It follows that $\lambda_{\varphi^n}(a)=\lambda_{\varphi^{n-1}}(\lambda_{\varphi}(a))$, and the claim follows by induction on $n$.

\newcommand{\T}{S}
Let $\T$ be a set and let $\psi\colon \T\rightarrow \T$ be a map. We say that $\psi$ is \textit{idempotent} if $\psi^2=\psi$. Note that if $x$ is a (purely) substitutive sequence produced by a substitution $\varphi$, then $x$ is also produced by the substitution  $\varphi^n$ for any integer $n\geq 1$. Similarly, the notions of $k$-automatic and $k^n$-automatic sequences coincide  \cite[Thm.\ 6.6.4]{AlloucheShallit-book}. For these reasons, we may freely replace $\varphi$ by $\varphi^n$, which will often have nicer properties. In Definition \ref{def:idemsub} we gather all the technical properties of the substitution that we intend to obtain in this manner. We will often need only some of these properties, but for simplicity we will not attempt to always state the precise minimal assumptions.

\begin{definition}\label{def:idemsub} A substitution $\varphi \colon \mathcal{A}\to \mathcal{A}^*$ is called \emph{idempotent} if it satisfies the following conditions: 
\begin{enumerate} 
\item \label{lem:propofsubs1} for all $a\in\mathcal{A}$ and $n\geq 1$ the set of letters appearing in $\varphi(a)$ is the same as the set of letters appearing in $\varphi^n(a)$;
\item\label{lem:propofsubs1'} for all $a\in\mathcal{A}$ and $n\geq 1$ the set of letters appearing  at least twice in $\varphi(a)$ is the same as the set of letters appearing  at least twice in $\varphi^n(a)$;
\item \label{lem:propofsubs2} for all $a\in \mathcal{A}$ the initial letter of $\varphi(a)$ is prolongable;
\item \label{lem:propofsubs3} the map $\lambda_{\varphi}\colon \mathcal{A}'\rightarrow \mathcal{A}'$  is idempotent.
\end{enumerate}
\end{definition} 

Note that if $\varphi\colon\mathcal{A}\rightarrow\mathcal{A}^*$ is an idempotent substitution, then for each $b\in\mathcal{A}$, $\mathcal{A}_b$ consists exactly of $b$ and the letters appearing in $\varphi(b)$.  Furthermore, a letter $b$ is ample if and only if $b$ appears in $\varphi(b)$ and it is very ample if and only if it appears at least twice in $\varphi(b)$.

\begin{lemma}\label{lem:elementary}
Let $\T$ be a finite set, and let $\psi\colon \T\rightarrow \T$ be a map. There exists an integer $m\geq 1$ such that $\psi^m$ is idempotent.
\end{lemma}
\begin{proof}
Take, for example, $m=|\T|!$. (This is a special case of a well-known fact that any finite semigroup has an idempotent, which itself is a very special case of the Ellis--Numakura Lemma \cite[Lem.\ 1]{Ellis58}.) 
\end{proof}

\begin{lemma}\label{lem:propofsubs}
Let $\varphi\colon \mathcal{A}\rightarrow \mathcal{A}^*$ be a substitution. There exists an integer $m\geq 1$ such that the substitution $\varphi^m$ is idempotent.
\end{lemma}
\begin{proof}

Note first that properties \eqref{lem:propofsubs1}, \eqref{lem:propofsubs1'}, \eqref{lem:propofsubs2} and \eqref{lem:propofsubs3} in Definition \ref{def:idemsub} are preserved after replacing $\varphi$ by its iterate. Thus, it is enough to find an appropriate integer $m$ for each of them separately and then take the final $m$ to be their common multiple. We will first choose $m$ so that properties \eqref{lem:propofsubs1} and \eqref{lem:propofsubs1'} hold for $\varphi^m$. For $a\in \mathcal{A}$ and a substitution $\varphi\colon \mathcal{A}\rightarrow \mathcal{A}^*$, let $S_{\varphi}(a)$ denote the set of letters appearing in $\varphi(a)$. Consider the map $\psi\colon 2^{\mathcal{A}}\rightarrow 2^{\mathcal{A}}$ that sends a subset $A$ of $\mathcal{A}$ to the set $\bigcup_{a\in A} S_{\varphi}(a)$.   Note that $\psi^n(\{a\})=S_{\varphi^n}(a)$ for all $n\geq 1$ and $a\in \mathcal{A}$.  By Lemma \ref{lem:elementary} there exists an integer $m\geq 1$ such that $\psi^m=\psi^{2m}$. This implies that $S_{\varphi^m}(a)=S_{\varphi^{nm}}(a)$ for all  $a\in \mathcal{A}$ and $n\geq 1$, and hence the substitution $\varphi^m$ satisfies property  \eqref{lem:propofsubs1}. We then obtain property \eqref{lem:propofsubs1'} by repeating the reasoning above with the set $2^{\mathcal{A}}$ replaced by the set $3^{\mathcal{A}}$, which includes the information on whether a letter appears in a word at least twice, exactly once or not at all. Thus, we may assume that properties  \eqref{lem:propofsubs1} and \eqref{lem:propofsubs1'} hold. 

To prove the remaining properties, let $\alpha_{\varphi}(a)$ denote the initial letter of $\varphi(a)$ for $a\in \mathcal{A}$. Note that $\alpha_{\varphi}^n=\alpha_{\varphi^n}$ for all $n\geq 1$. By Lemma \ref{lem:elementary}, there exists $m'\geq 1$ such that $\alpha_{\varphi}^{m'}=\alpha_{\varphi}^{2m'}$, and hence $\varphi^{m'}$ satisfies property \eqref{lem:propofsubs2}. A similar reasoning applied to the map $\lambda_{\varphi}\colon \mathcal{A'} \to \mathcal{A'}$ proves that some iterate of $\varphi^{m'}$ satisfies the remaining property \eqref{lem:propofsubs3}.\end{proof}

\section{Subsystems of substitutive systems}\label{sec:subsystems}

This  section studies subsystems of substitutive systems.  Recall that every substitutive system is a topological factor of a purely substitutive system $X$ (the factor map being given by a coding), and that we may assume that the substitutive sequence generating $X$ is produced by an idempotent substitution (see Lemma \ref{lem:propofsubs}). Our main result is Theorem \ref{thm:main} below, which says that transitive subsystems of substitutive systems are still substitutive. 
The proof of this result will occupy the whole section. We will first prove the statement for purely substitutive sequences  and obtain the general result by an easy reduction. Along the way, we will obtain a more detailed description of all transitive subsystems.

\begin{theorem}\label{thm:main}
Every transitive subsystem of a substitutive system is substitutive. Every transitive subsystem of a $k$-automatic system is $k$-automatic.
\end{theorem}

\subsection*{Minimal subsystems of substitutive systems} 

We start by investigating minimal subsystems of (purely) substitutive systems. Actually, we work in a slightly more general context of systems of the form $X_{\varphi}$ with $\varphi$ idempotent. We will show that in this case all minimal subsystems arise as $X_b$ for some minimal letter $b$. In particular, every minimal subsystem of a substitutive system is substitutive and  every minimal subsystem of a  $k$-automatic system is $k$-automatic.

\begin{proposition}\label{prop:minimalsubsystems}
Let $\varphi\colon \mathcal{A}\rightarrow \mathcal{A}^{*}$ be an idempotent substitution.  Let $Y$ be a subsystem of $X_{\varphi}$. Then $Y$ is minimal if and only if $Y=X_b$ for some minimal letter $b\in\mathcal{A}$.
\end{proposition}
\begin{proof} 
First we show that every system $X_b$ with $b$ minimal is minimal. If $b\in\mathcal{A}$ is a minimal letter, then the substitution  $\varphi|_{\mathcal{A}_b}\colon \mathcal{A}_b\rightarrow \mathcal{A}_b^{*}$ is primitive. Since every primitive substitution gives rise to a minimal system \cite[Prop.\ 5.5]{Queffelec-book}, $X_b$ is minimal.

Now assume that $Y$ is a minimal subsystem of $X_{\varphi}$. Fix an integer $m \geq 1$. Choosing a sufficiently long word $w\in \mathcal{L}(Y)$, we can find a letter $a\in\mathcal{A}$ such that $\varphi^m(a)$ appears in $w$, and hence $\varphi^m(a) \in \mathcal{L}(Y)$. Since $\mathcal{A}$ is finite, there is some letter $a\in \mathcal{A}$ such that $\varphi^m(a) \in \mathcal{L}(Y)$ for infinitely many $m$. Since $\varphi$ is idempotent, the set of letters appearing in $\varphi^l(a)$ is independent of  $l \geq 1$, and hence some minimal letter $b$ appears in $\varphi^{l}(a)$ for all $l \geq 1$. It follows that $\varphi^n(b)\in \mathcal{L}(Y)$ for all $n\geq 0$, and hence $X_b\subset Y$. By minimality of $Y$, we have $X_b=Y$. 
\end{proof}

\begin{corollary}\label{cor:minimalfinite}
 Let $X$ be a substitutive system. The number of minimal subsystems of $X$ is finite.
\end{corollary}
\begin{proof}
If $X$ is purely substitutive, then it is of the form $X=X_{\varphi}$ for some idempotent substitution $\varphi$, and the claim follows from Proposition \ref{prop:minimalsubsystems}. In the general case, write $X$ as a topological factor of a purely substitutive system and use Lemma \ref{lem:factor}.
\end{proof}

\begin{remark} A very special case of Proposition \ref{prop:minimalsubsystems} was proven in a different language in \cite[Lemma 2.3]{BK} by the first-named and the second-named author for constant length substitutions and one-point subsystems (and with a slightly weaker notion of idempotency). \end{remark} 

\subsection*{Transitive subsystems of substitutive systems} 

Let $\varphi\colon \mathcal{A}\rightarrow \mathcal{A}^{*}$ be a substitution. If $a$ is a (not necessarily minimal) letter in $\mathcal{A}$, then $X_a$ is a subsystem of $X_{\varphi}$. 
 It would be tempting to conjecture that all transitive subsystems of $X_{\varphi}$ are of this form. The following examples show that this is not the case. 
 
\begin{example}\label{infinite}\leavevmode \begin{enumerate}[wide]

\item \label{example1} Let $\mathcal{A}=\{0,1,2,3\}$ and let $\varphi\colon \mathcal{A}\rightarrow \mathcal{A}^*$ be the substitution given by
\[ \varphi(0)=12,\quad  \varphi(1)=11,\quad  \varphi(2)=23,\quad \varphi(3)=32.\]
Let $y$ denote the biinfinite sequence $y={}^{\omega}\!\varphi(1)\varphi^{\omega}(2)$. For an integer $n$  consider the suffix $y_{[n,\,\infty)}=y_n y_{n+1}\cdots$ of $y$. This is just the Thue--Morse sequence on the alphabet $\{2,3\}$ with $n$ first symbols removed if $n\geq 0$ or preceded by $1^{|n|}$ if $n<0$, and it lies in $X_{\varphi}$ since every factor of $y$ is a factor of some $\varphi^n(12)=\varphi^{n+1}(0)$. Consider the subsystems $Y_n=\overline{\Orb(y_{[n,\,\infty)})}\subset X_{\varphi}$. For $n\geq 0$, the system $Y_n$ is just the Thue--Morse system (since it is minimal), while for $n<0$ it is equal to the Thue--Morse system with $|n|$ extra points adjoined. Hence, for $n<0$ the systems $Y_n$ are pairwise distinct, and are different from each $X_b$ for $b\in \mathcal{A}$. In fact, $X_{\varphi}=\bigcup_{n\leq 0} Y_n \cup \{1^{\omega}\}$ and $X_0=X_{\varphi}$, $X_1=\{1^{\omega}\}$ and $X_2=X_3=Y_0$ (cf.\ Corollary \ref{cor:dichotomy} below and note that $X_2$ is minimal).

\item  \label{example2} Let $\mathcal{A}=\{0,1,2,3\}$ and let $\tau\colon \mathcal{A}\rightarrow \mathcal{A}^*$ be the substitution given by
\[\tau(0)=01023,\quad \tau(1)=12,\quad \tau(2)=22,\quad \tau(3)=33.
\] Write $v=01$ and $w=23$. Let $z$ denote the biinfinite sequence (indexed so that the $0$ below occurs at the $0$-th position)
\[z=\cdots\tau^2(v)\tau(v)v0w\tau(w)\tau^2(w)\cdots.\]
For an integer $n$ consider the suffix $z_{[n,\,\infty)}=z_nz_{n+1}\cdots$ of $z$. Every factor of $z_{[n,\,\infty)}$ is a factor of some $\tau^m(0)$, and hence $z_{[n,\,\infty)}$ lies in $X_{\tau}$. Consider the subsystems $Z_n=\overline{\Orb(z_{[n,\,\infty)})}\subset X_{\tau}$. It is easy to see that $Z_n= {\Orb(z_{[n,\,\infty)})}\cup\{3^k2^{\omega} \mid k\geq 0\}\cup \{2^k3^{\omega}\mid k\geq 0\}$, and hence the systems $Z_n$ are pairwise distinct and different from each $X_b$ for $b\in \mathcal{A}$.
\end{enumerate}
\end{example}

Now assume that $\varphi\colon \mathcal{A}\rightarrow \mathcal{A}^{*}$ is an idempotent substitution. The next proposition characterises points $y\in X_{\varphi}$ such  that $\overline{\Orb(y)}$ is not equal to any $X_b$ for $b\in\mathcal{A}$. We show that all such points are substitutive. Note that this is by no means obvious. In fact, substitutive systems have often continuum many points (e.g.\ the Thue--Morse system), while the number of substitutive sequences over a given alphabet is only countable.

\begin{proposition}\label{prop:dichotomy}
Let $\varphi\colon \mathcal{A}\rightarrow \mathcal{A}^{*}$ be an idempotent substitution. Let $y\in X_{\varphi}$ and let $Y$ be the orbit closure of $y$. Then at least  one of the following conditions holds: 
 \begin{enumerate}[label=\textup{(}a.\roman*\textup{)}]
\item\label{prop:dichotomya1} there exists a letter $a$ in $\mathcal{L}(X_{\varphi})$ such that $y\in X_a$;
\item\label{prop:dichotomya2} there exist a backwards prolongable letter $a$ and a prolongable letter $c$ such that $ac\in \mathcal{L}(X_{\varphi})$ and  $y$ is a suffix of ${}^{\omega}\!\varphi(a) \varphi^{\omega}(c)$. 
\end{enumerate}
Assume moreover that $Y$ is different from each $X_b$ for $b\in \mathcal{A}$. Then at least one of the following conditions holds: 
\begin{enumerate}[label=\textup{(}b.\roman*\textup{)}]

\item\label{prop:dichotomyb1} there exists a letter $a$ such that $\varphi(a)=v_a a w_a$ for some words $v_a$ and $w_a$ such that $w_a\neq \epsilon$, $w_a$ contains only letters $b$ such that $b\prec a$, and $y$ is a suffix of \[\cdots \varphi^2(v_a)\varphi(v_a) v_a a w_a \varphi(w_a) \varphi^2(w_a) \cdots; \]
\item\label{prop:dichotomyb2}  there exist a backwards prolongable letter $a$ and a prolongable letter $c$ such that $ac\in \mathcal{L}(X_{\varphi})$ and  $y$ is a suffix of ${}^{\omega}\!\varphi(a) \varphi^{\omega}(c)$ (i.e.\ the sequence $y$ satisfies condition \ref{prop:dichotomya2}).
\end{enumerate}\end{proposition}

\begin{proof}
Let $y^0=y$. Using Lemma \ref{lem:prefix}, we inductively construct for $i\geq 0$ letters $a_i$, sequences $y^i\in X_b$ with initial letters $a_i$, and  proper prefixes $u_i$ of  $\varphi(a_{i+1})$ such that \[\varphi(y^{i+1})=u_iy^i.\]  Note that since $a_i$ appears in $\varphi(a_{i+1})$, we have $a_{i+1} \succcurlyeq a_i$, and hence letters $a_i$ become equivalent for sufficiently large $i$. We will now show that $y$ satisfies one of the properties  \ref{prop:dichotomya1} and  \ref{prop:dichotomya2}.   We consider two cases.

{\bf Case I} (for infinitely many $i$ the length of $u_i$ is strictly smaller than $|\varphi(a_{i+1})|-1$). In this case for infinitely many $i$ the prefix of $y^i$ of length $2$ is a factor of $\varphi(a_{i+1})$. Let $a$ be a letter that occurs infinitely many times among $a_i$. Then $a$ is ample. Since $a$ is the initial letter of some $y^i$, it lies in  $ \mathcal{L}(X_{\varphi})$. Since $\varphi$ is growing and $\varphi(y^{i+1})=u_i y^i$, it follows that every prefix of $y$ is a factor of $\varphi^i(a)$ for some $i \geq 0$. In particular,  $y\in X_a$ and hence property \ref{prop:dichotomya1}                                                                                                                                                                                                                                                                                                                                                                                                                                                                                                                                                                                                                                                                                                                                                                                                                                                                                                                                                                                                                                                                                                                                                                                                                                                                                                                                                                                                                                                                                                                                                                                                                                                                                                                                                                                                                                                                                                                                                                                                                                                                                                                                                                                                                                                                                                                                                                                                                                                                                                                                                                                                                                                                                                                                                                                                                                                                                                                                                                                                                     holds.

{\bf Case II} (for all sufficiently large $i$ the length of $u_i$ is equal to $|\varphi(a_{i+1})|-1$). Let $i_0$ be such that we have $|u_i|=|\varphi(a_{i+1})|-1$ and $a_{i+1} \sim a_i$ for all $i\geq i_0$. Take $i\geq i_0$. Then $a_i$ is the final letter of $\varphi(a_{i+1})$, which implies that $a_{i+1}$ is ample and $\lambda(a_{i+1})=a_i$. Since the map $\lambda$ is idempotent, we have $a_{i+1}=a_i$ for all $i\geq i_0$. Put $a=a_{i_0}$, and note that $a$ is backwards prolongable. 

For $i\geq i_0$ the sequence $T(y^i)$ is the image of $T(y^{i+1})$ by $\varphi$. Iterating this for $i\geq i_0$, we see that for each $n\geq 0$ and $d=T(y^{i_0+n})_0$ the word $\varphi^n(d)$ is a prefix of $T(y^{i_0})$. Choose some letter $d$ that arises in this manner for infinitely many $n$, and put $c=\varphi(d)_0$. Since $\varphi$ is idempotent, $c$ is prolongable, and the assumption on $d$ shows that $T(y^{i_0})=\varphi^{\omega}(c)$. This shows that $y^{i_0}=a\varphi^{\omega}(c)$ and in particular $ac\in \mathcal{L}(X_{\varphi})$. Since $y$ is a suffix of $\varphi^{i_0}(y^{i_0})$, it is also a suffix of ${}^{\omega}\!\varphi(a)\varphi^{\omega}(c)$ and property  \ref{prop:dichotomya2} holds. This ends the proof of the first assertion.

Now assume that $Y$ is different from each $X_b$ for $b \in\mathcal{A}$. To show the second claim, we only need to treat Case I. We will show that in this case $y$ satisfies property \ref{prop:dichotomyb1}. (Recall that the conditions \ref{prop:dichotomyb2} and \ref{prop:dichotomya2} are the same.) As in the reasoning above, let $a$ be a letter that occurs infinitely many times among $a_i$, and recall that $a$ is ample, $a\in \mathcal{L}(X_{\varphi})$ and $y \in X_a$, whence $Y \subsetneq X_a$.

We claim that for sufficiently large $i$ the sequence $y^i$ contains no letters equivalent to $a$ at non-initial positions. Indeed, if $y^i$ contains a letter $c \sim a$ at a non-initial position, then $y$ contains $\varphi^i(c)$. If this happened for infinitely many $i$, 
the word $\varphi^n(a)$ would appear in $y$ for each $n \geq 0$, contradicting the assumption that $Y$ is a proper subset of $X_c = X_a$.

Since for sufficiently large $i$ the letters $a_i$ are all equivalent (and hence ample) and since the sequence $y^i$ contains no letters equivalent to $a$ at non-initial positions, we have $\lambda(a_{i+1})=a_i$ for sufficiently large $i$. Since the map $\lambda$ is idempotent, the sequence $a_i$ is eventually constant with value $a$. It follows that $\lambda(a)=a$ and $\varphi(a)=v_aaw_a$ with $w_a$ nonempty (since we are in Case I) and containing only letters $c\prec a$.  

Choose $i_0\geq 0$ so that for $i\geq i_0$ we have $a_i=a$  and the sequence $y^i$ contains no letters equivalent to $a$ at non-initial positions. Since $\varphi(y^{i+1})=u_iy^i$, we must have  
\[ y^{i_0} = a w_a \varphi(w_a)\varphi^2(w_a)\cdots.
\]
Hence $y$ is a suffix of  \[\varphi^{i_0}(y^{i_0})=\varphi^{i_0-1}(v_a) \cdots \varphi^{2}(v_a) \varphi(v_a) v_a a w_a \varphi(w_a)\varphi^2(w_a)\cdots.\qedhere\] 
\end{proof}

\begin{corollary}\label{cor:dichotomy}
Let $\varphi\colon \mathcal{A}\rightarrow \mathcal{A}^{*}$ be an idempotent substitution. Let $y\in X_{\varphi}$ and let $Y$ be the orbit closure of $y$. Assume that $Y$ is different from each $X_b$ for $b\in \mathcal{A}$. Then $y$ is a substitutive sequence. If furthermore $\varphi$ is a substitution of constant length $k$, then $y$ is $k$-automatic.
\end{corollary}
\begin{proof}
This follows immediately from Proposition \ref{prop:dichotomy}, Lemma \ref{lem:morvarphicseq} and the fact that substitutive (resp.,\ $k$-automatic) sequences are closed under backward and forward shifting.
\end{proof}

The next proposition characterises systems  $X_{\varphi}$ that are transitive.  

\begin{proposition} \label{prop:subsystemsXb}
Let $\varphi\colon \mathcal{A}\rightarrow \mathcal{A}^{*}$ be a substitution. Let $n\geq 1$ be such that $\varphi^n$ is idempotent. The following conditions are equivalent:
\begin{enumerate}
\item \label{prop:subsystemsXbi} $X_{\varphi}$ is transitive,
\item \label{prop:subsystemsXbii} $X_{\varphi}=X_{\varphi^n, b}$ for some letter $b\in \mathcal{A}$ that is either prolongable under $\varphi^n$ or very ample.
\end{enumerate}
Moreover, if $X_{\varphi}$ is transitive, then it is substitutive. Furthermore, if  $\varphi$ is a substitution of constant length $k$, then $X_{\varphi}$ is  $k$-automatic.
\end{proposition}

\begin{proof} Note first that under either of the assumptions \eqref{prop:subsystemsXbi} and  \eqref{prop:subsystemsXbii} we have $X_{\varphi}=X_{\varphi^n}$ (in the former case by Lemma \ref{lem:nthpower}, in the  latter case it is obvious). Hence, we may assume that $\varphi$ itself is idempotent and $n=1$.

We first show that \eqref{prop:subsystemsXbii} implies both \eqref{prop:subsystemsXbi} and the final claim. If $b$ is prolongable, then $X_\varphi = X_b$ is the orbit closure of $\varphi^\omega(b)$, from which all the remaining claims follow easily. Suppose that this is not the case. Then $b$ appears at least twice in $\varphi(b)$ and we can write $\varphi(b)=vbw$, where $v,w$ are words such that $b$ appears in $w$. The word $x_n=w\varphi(w)\cdots\varphi^n(w)$ is a suffix of $\varphi^{n+1}(b)$ and hence the sequence $x=w\varphi(w)\varphi^2(w)\cdots$ lies in $X_b$. Since $\varphi^n(b)$ is a factor of $x$ for all $n\geq 0$, the orbit of $x$ is dense in $X_b$. The sequence $x$ is substitutive by Lemma \ref{lem:morvarphicseq}. Furthermore, if $\varphi$ is a substitution of constant length $k$, then $x$ is $k$-automatic. 

It remains to prove that \eqref{prop:subsystemsXbi} implies \eqref{prop:subsystemsXbii}. Assume that $X_{\varphi}$ is transitive and let $y$ be a point in $X_{\varphi}$ with a dense orbit.    Since $X_{\varphi}=\bigcup_{b \in\mathcal{A}} X_{b}$, there exists $b \in \mathcal{A}$ such that $X_{\varphi}=X_b$; pick a minimal $b$ with this property. Suppose that $b$ is not prolongable  and not very ample. It means that $b$ appears in $\varphi(b)$ at most once, and at a non-initial position. 

{\bf Case I} ($b$ appears in $\varphi(b)$ exactly once, and at a non-initial and non-final position). Write $\varphi(b)=vbw$ for nontrivial words $v,w \in \mathcal{A}^*$. Since $\varphi$ is idempotent, by property \eqref{lem:propofsubs1'} in Definition \ref{def:idemsub} every word \[\varphi^n(b)=\varphi^{n-1}(v) \cdots \varphi(v) v b w \varphi(w) \cdots \varphi^{n-1}(w)\] contains exactly one occurrence of $b$, and hence every point in $X_b$ contains at most one occurrence of $b$. On the other hand, every suffix of the biinfinite sequence $\cdots \varphi^2(v) \varphi(v) v b w \varphi(w) \varphi^2(w) \cdots$ lies in $X_b$, and hence $X_b$ contains infinitely many points in which $b$ appears. It follows that $X_b$ is not transitive.

{\bf Case II} (either $b$ does not appear in $\varphi(b)$ or appears only at the final position). In this case $b\notin \mathcal{L}(X_b)$ and for all $a\in \mathcal{A}_b$ different from $b$ we have $a\prec b$. Applying  Proposition \ref{prop:dichotomy} to the system $X_b$, we see that either $y\in X_a$ for some $a\prec b$ or $y$ is a suffix of ${}^{\omega}\!\varphi(a)\varphi^{\omega}(c)$ for some backwards prolongable letter $a$ and prolongable letter $c$ such that $ac\in \mathcal{L}(X_b)$. The first case implies that $X_{\varphi}=X_a$ and contradicts the choice of $b$. In the second case $ac\in \mathcal{L}(X_b)$ implies that all suffixes of ${}^{\omega}\!\varphi(a)\varphi^{\omega}(c)$  lie in $X_b$. Since the orbit of $y$ is dense in $X_b$, for each $n\geq 1$ the sequence $\varphi^{n}(a)\varphi^{\omega}(c)$ has arbitrarily long prefixes in common with some forward shift of $y$. Since $\varphi^n(a)$ is a suffix of $\varphi^{n+1}(a)$ for each $n\geq 0$ and $y$ has the form $y=u\varphi^{\omega}(c)$ for some finite word $u$, all $\varphi^n(a)\varphi^n(c)$ are in fact factors of $\varphi^{\omega}(c)$. Letting $n$ tend to infinity, we conclude that all suffixes of ${}^{\omega}\!\varphi(a)\varphi^{\omega}(c)$ lie in $X_c$. In particular, $y\in X_c$, which again contradicts the choice of $b$. This ends the proof.\end{proof}

\begin{proof}[Proof of Theorem \ref{mainthmB}] The claim follows from Proposition \ref{prop:dichotomy} and the proof of Proposition \ref{prop:subsystemsXb}. More precisely, we choose $\tau$ to be a power of $\varphi$ such that $\tau$ is idempotent, and construct the required words $a,b,v,w$ and the sequence $y$ as in Proposition \ref{prop:dichotomy}(b) if the subsystem $Y$ is different from $X_{\tau,c}$ for all $c\in \mathcal{A}$ and as in the proof of Proposition \ref{prop:subsystemsXb} otherwise. (In the former case the claim holds for every $y$ whose orbit closure is $Y$.) \end{proof}

\begin{proof}[Proof of Theorem \ref{thm:main}]
The claim for transitive subsystems of systems of the form $X=X_{\varphi}$ for an idempotent substitution $\varphi\colon \mathcal{A} \to \mathcal{A}^*$ follows immediately from Corollary \ref{cor:dichotomy} and Proposition \ref{prop:subsystemsXb}.

 In general, if $Y$ is a transitive subsystem of a substitutive system $X$, we consider $X$ as a topological factor $X=\pi(X_{\varphi})$ of some  $X_{\varphi}$ for an idempotent substitution $\varphi$ and a coding $\pi$. Choose $y \in Y$ such that $Y=\overline{\Orb(y)}$ and let $z\in X_{\varphi}$ be such that $\pi(z)=y$. Put $Z=\overline{\Orb(z)}$.  By compactness we have  $\pi(Z)=Y$. Since $Z$ is a transitive subsystem of $X_{\varphi}$, it is substitutive, and hence so is the system $Y=\pi(Z)$. A similar argument proves the claim concerning $k$-automatic systems.
\end{proof}

\subsection*{Two-sided substitutive shifts and their subsystems}

We close this section with the remark that the results formulated above have their analogues for two-sided shifts.  For a substitution $\varphi\colon \mathcal{A}\rightarrow \mathcal{A}^*$ let $ X_\varphi^{\Z}$  denote the two-sided dynamical system generated by  $\varphi$, i.e.\
\[
X_{\varphi}^{\Z}=\{z\in {}^{\omega}\!\mathcal{A}^\omega \mid \text{ every factor of } z  \text{ appears in }\varphi^n(a) \text{ for some } n\geq 0 \textrm{ and } a\in\mathcal{A}\}.
\]
For a letter $a$, the system $X_a^{\Z}$ is defined accordingly.  A sequence $y=(y_n)_{n} \in {}^{\omega}\!\mathcal{A}^\omega$ is \textit{substitutive} if both $(y_n)_{n\geq 0}$ and $(y_n)_{n<0}$ are substitutive as one-sided sequences. This is obviously the same as saying that all (one-sided) prefixes and suffixes of $y$ are substitutive. Let $T$ denote the shift map on  ${}^{\omega}\!\mathcal{A}^\omega$. For two-sided systems we consider the two-sided orbit ${\Orb}^{\Z}(y) = \{T^n(y) \mid n\in \Z\}$ of a point $y$. A \emph{two-sided substitutive system} is the (two-sided) orbit closure of a two-sided substitutive sequence.  We define a two-sided $k$-automatic sequence and a two-sided $k$-automatic system in the same way. The main results for two-sided shifts are the same as or simpler than for the one-sided ones. The proofs are mutatis mutandis the same, and we present them in a briefer manner. The most notable difference between two-sided and one-sided shifts is that in the two-sided case every substitutive system has only finitely many subsystems.

\begin{theorem}\label{thm:main-Z}
Every transitive subsystem of a two-sided substitutive system is substitutive. Every transitive subsystem of a two-sided $k$-automatic system is $k$-automatic.
\end{theorem}

To prove this result, we first state three lemmas, which are analogous to the previously described results for one-sided systems. 
\begin{lemma}\label{lem:twosidquot} Every two-sided substitutive system $X$ arises as the image $X=\pi(X_{\varphi}^{\Z})$ of a transitive system $X_{\varphi}^{\Z}$ generated by a substitution $\varphi\colon \mathcal{A} \to \mathcal{A}^*$  via a coding $\pi\colon \mathcal{A}\to \mathcal{B}$. If $X$ is $k$-automatic, we may choose $\varphi$ to be of constant length $k$.\end{lemma}
\begin{proof} Let $X$ be a two-sided substitutive system arising as the orbit closure of a sequence $y=(y_n)_n$. Since  $(y_n)_{n\geq 0}$ and $(y_n)_{n<0}$ are one-sided  substitutive, we may find substitutions $\varphi_1 \colon \mathcal{A}_1\to \mathcal{A}_1^*$ and $\varphi_2 \colon \mathcal{A}_2\to \mathcal{A}_2^*$, codings $\pi_1 \colon \mathcal{A}_1 \to \mathcal{B}$ and $\pi_2 \colon \mathcal{A}_2 \to \mathcal{B}$, a prolongable letter $a_1\in \mathcal{A}_1$ and a backwards prolongable letter $a_2\in \mathcal{A}_2$ such that $y=\pi_2({}^{\omega}\!\varphi_2(a_2))\pi_1(\varphi_1^{\omega}(a_1))$. We may assume that $\mathcal{A}_1$ and $\mathcal{A}_2$ are disjoint. Define a new alphabet $\mathcal{A}=\mathcal{A}_1\cup\mathcal{A}_2 \cup \{\spadesuit\}$ with a new symbol  $\spadesuit \notin\mathcal{A}_1\cup\mathcal{A}_2$. Glue $\varphi_i$ and $\pi_i$ to maps $\varphi\colon \mathcal{A} \to \mathcal{A}^*$ and $\pi\colon\mathcal{A} \to \mathcal{B}$ by putting $\varphi|_{\mathcal{A}_i}=\varphi_i$, $\pi|_{\mathcal{A}_i}=\pi_i$, and $\varphi(\spadesuit)=a_2w$, where $w$ is a prefix of $\varphi^{\omega}(a_1)$ chosen to be of arbitrary length in the substitutive case and of length $k-1$ in the $k$-automatic case. It is easy to see that $X_{\varphi}^{\Z}$ is a transitive system generated by the sequence ${}^{\omega}\!\varphi_2(a_2)\varphi_1^{\omega}(a_1)$ and that $X=\pi(X_{\varphi}^{\Z})$.\end{proof}

The remaining two lemmas are two-sided analogues of Lemmas \ref{lem:prefix} and \ref{lem:nthpower}. The former of these lemmas is proven in \cite[Prop. 5.10]{BKM}.
\begin{lemma}\label{lem:prefixinv}
Let $\varphi\colon \mathcal{A}\rightarrow \mathcal{A}^*$ be a substitution and let $x\in X^{\Z}_{\varphi}$. There exists $y\in X^{\Z}_{\varphi}$ such that $x=T^{l}(\varphi(y))$ for some integer $l \in [0, |\varphi(y_0)|-1]$.
\end{lemma}
\begin{lemma}\label{lem:nthpowerinv}
Let $\varphi\colon \mathcal{A}\rightarrow \mathcal{A}^*$ be a substitution. If $X^{\Z}_{\varphi}$ is transitive, then $X^{\Z}_{\varphi}=X^{\Z}_{\varphi^n}$ for any $n\geq 1$.
\end{lemma}
The proofs of these lemmas are analogous to those of  Lemmas \ref{lem:prefix} and \ref{lem:nthpower}. The main result is derived via essentially the same reasoning as before from Proposition \ref{prop:dichotomy-Z} and \ref{prop:subsystemsXb-Z} below.

\begin{proposition}\label{prop:dichotomy-Z}
Let $\varphi\colon \mathcal{A}\rightarrow \mathcal{A}^{*}$ be an idempotent substitution. Let $y\in X_{\varphi}^{\Z}$ and let 
\[ Y = \{z \in  X_{\varphi}^{\Z}\mid \text{every factor of $z$ appears in $y$}\} 
\]
be the orbit closure of $y$. Then at least  one of the following conditions holds: 
\begin{enumerate}[label=\textup{(}\roman*\textup{)}]
\item\label{prop:dichotomya1-Z} there exists $a\in \mathcal{L}(Y)$ such that $Y = X_a^{\Z}$;
\item\label{prop:dichotomya2-Z} there exist a backwards prolongable letter $a$ and a prolongable letter $c$ such that $ac\in \mathcal{L}(y)$ and $y$ is a shift of ${}^{\omega}\!\varphi(a) \varphi^{\omega}(c)$. 
\end{enumerate}
In particular, the number of subsystems of $ X_{\varphi}^{\Z}$ is finite.
\end{proposition}
\begin{proof}
Let $y^0=y$. Repeating the reasoning from the proof of Proposition \ref{prop:dichotomy}, we construct letters $c_i$, sequences $y^i \in  X_{\varphi}^{\Z}$ with initial letters  $c_{i}$ and integers $l_i\in [0, |\varphi(c_{i+1})|-1]$ such that 
\[
	y^{i} = T^{l_i} \varphi(y^{i+1})
\]
for each $i \geq 0$. We consider two cases depending on the asymptotic behaviour of $l_i$.

Assume first that $l_i = 0$ for all sufficiently large $i$ and put $a_i=y^{i}_{-1}$. For sufficiently large $i$, $c_i$ is the initial letter of  $\varphi(c_{i+1})$ and $a_i$ is the final letter of $\varphi(a_{i+1})$. Since $\varphi$ is idempotent, it follows from properties \eqref{lem:propofsubs2} and \eqref{lem:propofsubs3} in Definition \ref{def:idemsub} that  $a_i$ and $c_i$ are eventually constant, say $a_i = a$ and $c_i = c$ for $i$ sufficiently large. It follows that $c$ is prolongable, $a$ is backwards prolongable, and $y$ is a shift of ${}^{\omega}\!\varphi(a) \varphi^{\omega}(c)$. Similarly, if we assume that $l_i = |\varphi(c_{i+1})|-1$ for all sufficiently large $i$, then we may apply the same reasoning with $T(y^i)$ in place of $y^i$. 

Now assume that $l_i \neq 0$ and $l_i \neq |\varphi(c_{i+1})|-1$ for infinitely many $i$'s. Let $a$ be a letter that occurs infinitely many times among $c_i$. Then $a\in \mathcal{L}(Y)$ and  we can find arbitrarily large $j$ such that $\varphi^j(a) = y_{[n_j,m_j-1]}$, where $n_j \to -\infty$ and $m_j \to + \infty$ as $j \to \infty$. It follows that $Y = X_a^{\Z}$.

It follows immediately from our claim that $X_{\varphi}^{\Z}$  has only finitely many transitive subsystems, and hence finitely many  subsystems.
\end{proof}

\begin{remark}\label{remark:twosided} Two-sided substitutive systems have been considered by Maloney and Rust, mostly under different assumptions on the substitution, namely that it is recognisable and tame (for the definition, see  \cite[Definition 2.4]{MR}). Note that all growing substitutions are tame, but not all growing substitutions are recognisable. Under these assumptions, the finiteness of the number of subsystems of $X_{\varphi}^{\Z}$ follows from  \cite[Lemma 5.13]{MR}. The authors work with the tiling space $\Omega_{\varphi}$ (see \cite[Section 1.3]{MR}) associated with a substitution $\varphi$ and prove that the number of closed unions of path components of $\Omega_{\varphi}$ is finite. Since there is a bijective correspondence between subsystems of $X_{\varphi}^{\Z}$ and closed unions of path components of $\Omega_{\varphi}$, the claim follows. Finiteness of the number of minimal subsystems of two-sided substitutive systems if the substitution is either aperiodic or growing has also been observed by Bezuglyi--Kwiatkowski--Medynets \cite[Prop.\ 5.6 and Remark 5.7]{BKM}.
\end{remark}

For simplicity we only state the following result for idempotent substitutions, but the more general analogue can readily be derived in the same way as in Proposition \ref{prop:subsystemsXb}.

\begin{proposition}\label{prop:subsystemsXb-Z}
Let $\varphi\colon \mathcal{A}\rightarrow \mathcal{A}^{*}$ be an idempotent substitution. The following conditions are equivalent:
\begin{enumerate}
\item \label{prop:subsystemsXbi-Z} $X_{\varphi}^{\Z}$ is transitive;
\item \label{prop:subsystemsXbii-Z} one of the following conditions holds:
 \begin{enumerate}[label=\textup{(}\alph*\textup{)}]
\item \label{it:54:A} $ X_{\varphi}^{\Z}=X_{b}^{\Z}$ for some letter $b\in \mathcal{A}$ that is either very ample or ample but neither prolongable nor backwards prolongable;
\item \label{it:54:B} $X_{\varphi}^{\Z} = X_a^{\Z} \cup X_c^{\Z} \cup {\Orb}^{\Z}({}^{\omega}\!\varphi(a)\varphi^{\omega}(c))$ for some backwards prolongable letter $a$ and prolongable letter $c$.
\end{enumerate} 
\end{enumerate}
Moreover, if $ X_{\varphi}^{\Z}$ is transitive, then it is substitutive. Furthermore, if $\varphi$ is a substitution of constant length $k$, then $ X_{\varphi}^{\Z}$ is $k$-automatic.
\end{proposition}
\begin{proof}

Suppose \eqref{prop:subsystemsXbii} holds. Consider first the case \eqref{prop:subsystemsXbii-Z}.\ref{it:54:A}. If $\varphi(b)=vbw$ with $v,w \in \mathcal{A}^{*}$ nonempty, then $X_b^{\Z}$ is the orbit closure of the point 
\[ y = \cdots \varphi^2(v) \varphi(v) v b w \varphi(w) \varphi^2(w) \cdots.\] By Lemma  \ref{lem:morvarphicseq}, both one-sided sequences $w \varphi(w) \varphi^2(w) \cdots$ and  $v\varphi(v) \varphi^2(v) \cdots$ are substitutive, and thus $y$ is substitutive as well. Otherwise, by \eqref{prop:subsystemsXbii-Z}.\ref{it:54:A} we have $\varphi(b)=bvb$ for some $v\in \mathcal{A}^*$, and $\varphi^2(b)=bvb\varphi(v)bvb$. Thus, the previous property is satisfied for $\varphi^2$ and the claim follows from the equality $X_{\varphi}^{\Z}=X_{\varphi^2}^{\Z}$, which holds since $\varphi^n(b)$ is a prefix of $\varphi^{n+1}(b)$ for all $n\geq 0$. Finally, in case \eqref{prop:subsystemsXbii-Z}.\ref{it:54:B} the system $X_{\varphi}^{\Z}$ is the orbit closure of $y = {}^{\omega}\!\varphi(a)\varphi^{\omega}(c)$.

It remains to show that \eqref{prop:subsystemsXbi} implies \eqref{prop:subsystemsXbii}. Assume that $X_{\varphi}^{\Z}$ is transitive and let $y$ be a point in $X_{\varphi}^{\Z}$ with a dense orbit. Since $X_{\varphi}^{\Z}=\bigcup_{b \in\mathcal{A}} X_{b}^{\Z}$, there exists $b \in \mathcal{A}$ such that $ X_{\varphi}^{\Z}=  X_b^{\Z}$; pick minimal $b$ with this property. Suppose that $b$ does not satisfy the conditions in \eqref{prop:subsystemsXbii-Z}.\ref{it:54:A}, and so $b$ appears at most once in $\varphi(b)$, either at the initial or final position. Then $b\notin \mathcal{L}(X_b^{\Z})$ and $a \prec b$ for all $a \in \mathcal{A}_b \setminus \{b\}$. In particular, $ X_b^{\Z} \neq X_a^{\Z}$ for all $a \in \mathcal{L}(X_b^{\Z})$, so Proposition \ref{prop:dichotomy-Z} implies that there exists $ac \in \mathcal{L}( X_b^{\Z})$ such that $y$ is up to a shift equal to ${}^{\omega}\!\varphi(a)\varphi^{\omega}(c)$. Hence, case \eqref{prop:subsystemsXbii-Z}.\ref{it:54:B} holds.
\end{proof}
\begin{proof}[Proof of Theorem \ref{thm:main-Z}] The proof is analogous to the proof of Theorem \ref{thm:main}, replacing the use of Propositions \ref{prop:dichotomy} and \ref{prop:subsystemsXb} by Propositions \ref{prop:dichotomy-Z} and \ref{prop:subsystemsXb-Z}.
\end{proof}

\section{Finitary version of Cobham's theorem}\label{sec:cobham}

In this section, we prove a finitary generalisation of Cobham's  classical theorem. We  work again in the context of one-sided sequences.  Recall that integers $k,l \geq 2$ are \textit{multiplicatively independent} if $\log k/ \log l \in {\R}\setminus{\Q}$, i.e.\ $k$ and $l$ are not both integer powers (equivalently, rational powers) of the same integer. Let $k,l \geq 2$ be multiplicatively independent integers. The celebrated theorem of Cobham characterises sequences that are simultaneously $k$-automatic and $l$-automatic: these are precisely the sequences that are ultimately periodic. The following theorem provides a complete characterisation of sets of words that can occur as common factors of automatic sequences defined over multiplicatively independent bases.

\begin{theorem}\label{thm:maincobham}
Let $k,l\geq 2$ be multiplicatively independent integers, let $\mathcal{A}$ be an alphabet, and let $U\subset \mathcal{A}^*$. The following conditions are equivalent:
\begin{enumerate}
\item\label{thm:maincobhami} there exist a $k$-automatic sequence $x$ and an $l$-automatic sequence $y$ such that $U$ is  the set of common factors of $x$ and $y$;
\item\label{thm:maincobhamii}  the set $U$ is a nonempty finite union of sets of the form  $\mathcal{L}(^{\omega}vuw^{\omega})$, where $u,v,w$ are (possibly empty) words over $\mathcal{A}$.
\end{enumerate}
\end{theorem}

Theorem \ref{thm:maincobham} immediately implies Cobham's theorem and its strengthening due to Isabelle Fagnot  \cite{Fagnot-1997}, stated below. Unfortunately, this does not give an independent proof of Cobham's theorem since we will use Fagnot's result in the proof.

\begin{theorem}[Fagnot]\label{thm:languagecobham}
Let $k,l\geq 2$ be multiplicatively independent integers. Let $x$ be a $k$-automatic sequence and let $y$ be an $l$-automatic sequence. If $\mathcal{L}(x)=\mathcal{L}(y)$, then both sequences $x$ and $y$ are ultimately periodic.
\end{theorem}

From the result above and Theorem \ref{thm:main} we get the following corollary.

\begin{corollary}\label{cor:cobpoints}
Let $k,l \geq 2$ be multiplicatively independent integers and let $\mathcal{A}$ be an alphabet. Let  $X, Y \subset \mathcal{A}^{\omega}$ be subsystems such that $X$ is $k$-automatic and  $Y$ is $l$-automatic. If a sequence $z$ belongs to both $X$ and $Y$, then it is ultimately periodic.
\end{corollary}
\begin{proof}
Let $Z$ be the orbit closure of $z$. Then $Z$ is a transitive subsystem of both systems $X$ and $Y$. By Theorem \ref{thm:main}, there exist a $k$-automatic sequence $x\in X$ and an $l$-automatic sequence $y\in Y$ such that $Z$ is the orbit closure of $x$ and $y$, and hence the sequences $x$ and $y$ have the same language. By Theorem \ref{thm:languagecobham}, the system $Z$ is finite, and hence the sequence $z$ is ultimately periodic.
\end{proof}

The problem of describing common factors of automatic sequences was considered in \cite{MRSS}. The authors obtained, among other things, an upper bound on the length of a common prefix of aperiodic automatic sequences defined over multiplicatively independent bases in terms of the number of states of the automata generating the sequences. They further asked about the structure of the set of common factors of automatic sequences defined over multiplicatively independent bases.   
Since every ultimately periodic sequence is $k$-automatic for all integers $k\geq 2$, it is clear that we can get common factors of the form $vu^n$ for some words $v$, $u$ and arbitrarily large $n$. The following example shows that  common factors can be somewhat more complicated.

\begin{example} Let $\mathcal{A}=\{0,1,2\}$. Consider the $3$-automatic sequence $x= \varphi^\omega(0)$ produced by the substitution  $\varphi\colon \mathcal{A} \rightarrow \mathcal{A}^*$ given by
\[
\varphi(0)=012,\
\varphi(1)=111,\
\varphi(2)=222
\]
and the $4$-automatic sequence $y= \tau^\omega(0)$ produced by the substitution  $\tau\colon \mathcal{A} \rightarrow \mathcal{A}^*$ given by
\[
\tau(0)=0121,\
\tau(1)=1111,\
\tau(2)=2222.
\] Then \[x=0121^3 2^3 1^9 2^9 1^{27} 2^{27}\cdots \] and hence \[X_{\varphi}=\Orb(x)\cup \{2^n1^{\omega}\mid n\geq 0\}\cup \{1^n2^{\omega}\mid n\geq 0 \}.\] Similarly,  \[y=0121^5 2^4 1^{20} 2^{16} 1^{80} 2^{64}  \cdots \] and hence \[X_{\tau}=\Orb(y)\cup \{2^n1^{\omega}\mid n\geq 0 \}\cup \{1^n2^{\omega}\mid n\geq 0\}.\] The common factors of $x$ and $y$ are exactly the words in $\mathcal{L}({}^{\omega}12^{\omega})\cup \mathcal{L}({}^{\omega}21^{\omega})\cup \mathcal{L}(0121^3)$.
\end{example}

We will use Corollary \ref{cor:cobpoints} to show that  common factors of automatic sequences defined over multiplicatively independent bases are all of the form suggested by the example above. We need to introduce some additional terminology. Let $\mathcal{A}$ be an alphabet and let $x$ be a sequence over $\mathcal{A}$. Let $X$  denote the orbit closure of $x$. We say that a  factor $u$ of $x$  is  \textit{cyclic} if $u$ is nonempty and the periodic sequence $u^{\omega}$ lies in  $X$. We say that $u$ is \textit{primitive} if it is not of the form $u=v^n$ for some $v\in \mathcal{A}^*$ and $n\geq 2$. Since the orbit closure of a periodic sequence is minimal, it follows from Corollary \ref{cor:minimalfinite} that the set of primitive cyclic factors of a substitutive sequence is finite. We say that a common factor  of sequences $x,y\in\mathcal{A}^{\omega}$ is \textit{cyclic} if it is cyclic as a factor of $x$ and as a factor of $y$.

\begin{remark}\label{remark:substitutive} We cannot hope for a straightforward generalisation of Theorem \ref{thm:maincobham} to the class of substitutive sequences. Recall that with every substitution $\varphi\colon\mathcal{A}\rightarrow\mathcal{A}^*$ we can associate a matrix  $M=(m_{ij})_{i,j\in \mathcal{A}}$, where $m_{ij}$ is the number of occurrences of the letter $i$ in the word $\varphi(j)$. By the Frobenius--Perron theorem $M$ always has a dominant eigenvalue $\lambda>0$. The eigenvalue $\lambda$ plays the role of a base for a substitutive sequence $x$ produced by $\varphi$ (for details see, e.g.\ \cite{DR-2011}), which allows us to define the class of $\lambda$-substitutive systems. The reason why Theorem \ref{mainthmC} fails in this setting is that  transitive subsystems of $\lambda$-substitutive systems need not be $\lambda$-substitutive. Consider  the following example. Let $\mathcal{A}=\{0,1\}$ and $\mathcal{B}=\{0,1,2,3\}$, let  $x= \varphi^\omega(0)$ be the sequence produced by the substitution  $\varphi\colon \mathcal{A} \rightarrow \mathcal{A}^*$ given by
\[
\varphi(0)=01,\
\varphi(1)=10,
\]
and let $y= \tau^\omega(2)$ be the sequence produced by the substitution  $\tau\colon \mathcal{B} \rightarrow \mathcal{B}^*$ given by
\[
\tau(0)=01,\
\tau(1)=10,\
\tau(2)=203,\
\tau(3)=3233.
\]
The systems $X=\overline{\mathcal{O}(x)}$  and $Y=\overline{\mathcal{O}(y)}$ are $2$-substitutive and $(2+\sqrt{2})$-substitutive, respectively, and the set of common factors of $x$ and $y$ consists precisely  of the factors of $x$.
\end{remark}

\subsection*{Occurrences of cyclic factors in automatic sequences}

Let $x$ be a $k$-automatic sequence. To proceed with the proof of Theorem \ref{thm:maincobham}, we first need to understand the structure of sets of the form \[S_x=\{n\geq 0\mid vu^nw \textrm{ is a factor of } x\}\] for fixed words $v$, $w$ and $u$. This is only interesting if $u$ is a cyclic factor of $x$, since otherwise $S_x$ is finite. If  $v$ is a suffix of some power of $u$ or $w$ is a prefix of some power of $u$, then the set $S_x$ is easy to determine, and either consists of all integers, or is finite and consists of all integers smaller than some constant (see Remark \ref{rem:effectiveS} below). Assume conversely that $v$ is not a suffix of any power of $u$ and $w$ is not a prefix of any power of $u$. We will show that the set $S_x$  is up to a finite set a finite union of translates of geometric progressions, and deduce that for automatic sequences $x$ and $y$ defined over multiplicatively independent bases the set $S_x\cap S_y$ is finite.

The problem above was also considered by Fagnot in \cite[Proposition 8]{Fagnot-1997} in the special case when the sequence $x$ takes values in $\{0,1\}$ and $v=w=1$, $u=0$. This result was used  to show that if $x$ is a $k$-automatic sequence, $y$ is an $l$-automatic sequence, $k,l\geq 2$ are multiplicatively independent, and $\mathcal{L}(x) \subset \mathcal{L}(y)$, then either $x$ contains only finitely many $1$'s or $1$'s occur in $x$ with bounded gaps (see \cite[Corollaire 10]{Fagnot-1997}). This is the crucial step in her proof of Theorem \ref{thm:languagecobham}.  It would be interesting to see if the general statement below could be reduced to the special case considered by Fagnot, but we found no such reduction. We give a proof of the general result that uses similar ideas as the one of Fagnot but seems quite different in details, and we deduce a much stronger finiteness result in Corollary \ref{cor:commonfactors}. We also discuss the question of effectiveness.

\begin{theorem}\label{prop:recurrence}
Let $k\geq 2$ be an integer. Let $x$ be a $k$-automatic sequence over an alphabet $\mathcal{A}$. Let $u,v,w$ be nonempty words over $\mathcal{A}$. Assume that $v$ is not a suffix of $u^n$ and $w$ is not a prefix of  $u^n$ for any integer $n$.  Let $S=\{n \geq 0\mid vu^nw \textrm{ is a factor of } x\}$. The set $S$ is a finite union of sets of the form $\{ak^{mn}+b\mid n\geq 0\}$ for some $a,b\in\Q$ and $m\in \N$ with $a, b \geq 0$, $a+b\in \Z$ and $(k^m-1)a\in\Z$.
\end{theorem}
\begin{proof} We begin the proof with a  few reductions.

{\bf Step I} (reduction to the case when $|u|=|v|=|w|$). First, we show that it is enough to prove the claim under the additional assumption that $|v|\leq |u|$ and $|w|\leq |u|$. Let $j>0$ be an integer such that $\max(|v|,|w|) \leq |u|^{j/2}$. Write 
\[ S_i =\{n\in S \mid n\equiv i \pmod j\}.\]
Then $S=\bigcup_{i=0}^{j-1} S_i$, and hence it is enough to show the claim separately for each of the sets $S_i$. For an integer $i\in [0,j-1]$, we put $u'=u^j$,  $ v'=vu^{\lfloor (i+1)/2\rfloor}$ and $w'=u^{\lfloor i/2\rfloor}w$. We then have $\max(|v'|,|w'|) \leq|u'|$ and 
\[S_i=\{n=jm+i\mid  v'(u')^m w' \textrm{ is a factor of } x\}.\]  
In order to further obtain $|v'|=|w'|=|u'|$, we consider all possible prolongations of the words $v'$ and $w'$ to words $v''$ and $w''$ of length $|u'|$ and such that $v'$ is a suffix of $v''$ and $w'$ is a prefix of $w''$. Every element of $S_i$ lies in one of the sets \[\{n=jm+i\mid  v''(u')^m w'' \textrm{ is a factor of } x\}\] for some choice of $v''$ and $w''$, except for the values $n\in S_i$ corresponding to factors $v'(u')^m w'$ which occur only at starting positions $<|u'|-|v'|$. Since there are only finitely many such values of $m$ (at most one for each starting position), we may assume that $|u|=|v|=|w|$.

{\bf Step II} (reduction to the case when $|u|=|v|=|w|=1$). Write $\ell$ for the common length of $u$, $v$ and $w$. We will now show that we may assume that $\ell=1$ by changing the alphabet. For an integer $i\in[0,\ell-1]$ let
\[\tilde{S}_i=\{n\in S\mid vu^nw=x_{[m,\, m+(n+2)\ell)} \textrm{ for some }  m\equiv i\pmod \ell\}
\]
denote the set of all integers $n\in  S$ such that the factor $vu^nw$ occurs in $x$ at a position $m\equiv i \pmod \ell$. Clearly, $S$ is the union of the sets $\tilde{S}_i$. 

Let $\mathcal{A}^\ell$ denote the set of words of length $\ell$ over $\mathcal{A}$. 
 Identifying words $u,v,w$  with letters $u',v',w'\in \mathcal{A}^\ell$ and the sequence $x$ with the corresponding sequence $x'\in (\mathcal{A}^\ell)^{\omega}$, we see that the set $\tilde{S}_0$ is equal to the set of all integers $n$ such that $v'(u')^nw'$ is a factor of $x'$.  The same reasoning applied to the sequence $T^i(x)$  instead of $x$ shows that the set $\tilde{S}_i$  is equal to the set of all integers $n$ such that $v'(u')^nw'$ is a factor of the $k$-automatic sequence $(T^i(x))'$. This allows us to assume that $u,v,w$ are single letters.
 
 {\bf Step III} (restating the problem in terms of purely automatic sequences). 
 Write the sequence $x$ as the image of a purely $k$-automatic sequence $y$ produced by a substitution $\varphi\colon\mathcal{B}\rightarrow\mathcal{B}^*$ of constant length $k$ under a coding $\pi\colon \mathcal{B}\rightarrow \mathcal{A}$. 
Let $T$, $C$ and $D$ denote the preimages of the letters $u$, $v$ and $w$ under the coding $\pi$, respectively. Note that $C\cap T=D\cap T=\emptyset$. The set $S$ can be expressed in terms of the sequence $y$ as 
\[
S=S(C,D,T)=\{n\geq 0\mid cwd \textrm{ is a factor of } y \textrm{ for some } c\in C, d\in D \textrm{ and } w\in T^* \text{ with } |w|=n \}.
\] We will prove that sets $S=S(C,D,T)$ satisfy the claim for all purely $k$-automatic sequences $y$ over an alphabet $\mathcal{B}$ and subsets $T,C,D\subset \mathcal{B}^*$ with $C\cap T=D\cap T=\emptyset$. Dividing $S$ into a finite union, we may further assume that the sets $C$ and $D$ consist of  single letters $c,d\in\mathcal{B}\setminus T$, and we write $S(c,d,T)$ for $S(\{c\},\{d\},T)$.

{\bf Step IV} (constructing a recurrence for the set $S$). For $m\geq 1$ by abuse of notation we let $\varphi^{-m}(T)$ denote the set $\varphi^{-m}(T^*)\cap \mathcal{A}$, i.e.\ the set of letters $a\in\mathcal{A}$ such that $\varphi^m (a)\in T^*$. We write $\mathcal{A}_T=\mathcal{A}\setminus \varphi^{-1}(T)$ for the set of letters $a\in\mathcal{A}$ such that $\varphi(a)\notin T^*$. For a letter $a\in \mathcal{A}_T$ let $\alpha(a)$  denote the first letter in $\varphi(a)$ that is not in $T$ and let $\omega(a)$ denote the last letter in $\varphi(a)$ that is not in $T$. We replace the substitution $\varphi$ by its power in order to get the property  $\varphi^{-1}(T)=\varphi^{-2}(T)$ (this is possible by Lemma \ref{lem:elementary}); note that this involves replacing $k$ by its power.

With every pair $(a,a')\in \mathcal{A}_T^2$ we associate an integer $q(a,a')$ in the following way. Write $\varphi(a)=v\omega(a)w$ and $\varphi(a')=w'\alpha(a')v'$ for some $w,w'\in T^*$ and $v,v'\in\mathcal{A}^*$. Put $q(a,a')=|w|+|w'|$. Consider the sets  \[\Omega_c=\{ a\in\mathcal{A}_T \mid \omega(a)=c\} \quad \text{and} \quad \mathit{A}_d=\{ a\in\mathcal{A}_T \mid \alpha(a)=d\}.\] Let $E\subset S$ denote the set of integers $n\in S$  such that $n<k-1$. We claim that
\begin{equation}\label{union}
S(c,d,T)=\bigcup_{(c',d')\in \Omega_c\times \mathit{A}_d} \left( kS(c',d',\varphi^{-1}(T))+q(c',d') \right) \cup E.
\end{equation} Let $n\geq k-1$. By definition, $n$ lies in  $S$ if and only if there exists a word $w\in T^*$ with $|w|=n$ such that $cwd$ is a factor of $y$. Since $\varphi(y)=y$ and $|cwd|\geq k+1$, this happens if and only if there exist a pair $(c',d')\in \Omega_c\times \mathit{A}_d$ and a word $w'\in (\varphi^{-1}(T))^*$ such that $cwd$ is a factor of $\varphi(c'w'd')$. It follows that $n\in S$ if and only if $n\in k S(c',d',\varphi^{-1}(T))+q(c',d')$ for some $(c',d')\in \Omega_c\times \mathit{A}_d$, with the `if' claim not requiring the assumption that $n\geq k-1$. This proves $\eqref{union}$.

Observe that \eqref{union} implies that it is enough to prove the claim for each of the sets $S(c',d',\varphi^{-1}(T))$. By the assumption on $\varphi$, we have that  $\varphi^{-1}(\varphi^{-1}(T))=\varphi^{-1}(T)$, and hence it is enough to show the claim for sets of the form $S=S(c,d,T)$ under the additional assumption that $\varphi^{-1}(T)=T$.  It follows that $\alpha(\mathcal{A}_{T}) \subset \mathcal{A}_{T}$ and $\omega(\mathcal{A}_{T}) \subset \mathcal{A}_{T}$. Writing temporarily $\alpha_{\varphi}$ and $\omega_{\varphi}$ for the maps $\alpha, \omega \colon \mathcal{A}_{T} \to \mathcal{A}_T$ defined with respect to the substitution $\varphi$, we note that $\alpha_{\varphi}^n=\alpha_{\varphi^n}$ and $\omega_{\varphi}^n=\omega_{\varphi^n}$ for all $n\geq 1$. Another application of Lemma \ref{lem:elementary} shows that after replacing $\varphi$ by an appropriate  power we get that the maps $\alpha$ and $\omega$ are idempotent, which we henceforth assume.

If the set $S$ is finite, the claim is obvious, so assume that $S$ is infinite. By \eqref{union}, the sets $\Omega_c$ and $\mathit{A}_d$ are nonempty, and hence (since $\omega$ and $\alpha$ are idempotent) we have $\omega(c)=c$, $\alpha(d)=d$. 

We now consider a family of recurrence sequences. For an element $r \in S$, consider the sequence $(n_t^r)_{t\geq 0}$ given by the formula
\begin{align*} n_0^r&=r, \\ n_t^r&= kn_{t-1}^r+q(c,d), \quad t\geq 1.\end{align*}
By \eqref{union} it is clear that $n_t^r \in S$ for all $r\in S$ and $t\geq 0$. We claim that \begin{equation}\label{eqn:formS} S = \{ n_t^r \mid t\geq 0, r\in S, 0\leq r <k^2-1\}.\end{equation}
This will end the proof of the claim, since the recurrence sequence $(n^r_t)_{t\geq 0}$ has the closed form $n_t^r=ak^{t}+b$ with $a=r+q(c,d)/(k-1)$ and $b=-q(c,d)/(k-1)$ satisfying $a+b\in \mathbf{Z}$ and $(k-1)a\in\mathbf{Z}$. (Note that in the process of the proof we have replaced the substitution $\varphi$ by its iterate, which has the effect of replacing the original  $k$ by its power.)

{\bf Step V} (proving the formula \eqref{eqn:formS}). We have already remarked that all $n_t^r$ are elements of $S$. For the converse claim (with the extra statement that one can take $r<k^2-1$), we will inductively apply \eqref{union}, which takes a simpler form since $\varphi^{-1}(T)=T$. 

Choose $m \in S$ and write $c_0=c$, $d_0=d$, $m_0=m$. If $m\geq k-1$, then by \eqref{union} we may write $m_0=km_1 + q(c_1,d_1)$ for some $c_1, d_1 \in \mathcal{A}_T$ with $\omega(c_1)=c$, $\alpha(d_1)=d$ and $m_1 \in S(c_1,d_1,T)$. If $m_1 \geq k-1$, we may repeat this procedure. In this way, we inductively construct sequences $(c_i)_{0\leq i \leq s}$,  $(d_i)_{0\leq i \leq s}$ and  $(m_i)_{0\leq i \leq s}$ with $c_i, d_i \in \mathcal{A}_T$, $\omega(c_{i+1})=c_i$, $\alpha(d_{i+1})=d_i$ and $m_i \in S(c_i,d_i,T)$ with $m_i=km_{i+1} + q(c_{i+1},d_{i+1})$. Furthermore, we have $m_i \geq k-1$ for $i<s$ and $m_i<k-1$ for $i=s$. 

Since the maps $\omega$ and $\alpha$ are idempotent, the conditions on $(c_i)$ and $(d_i)$ imply that $c_i=c$ and $d_i=d$ for $i\in[0,s-1]$ (but not necessarily for $i=s$). This shows that  $m_i=km_{i+1} + q(c,d)$ for $i\in[0,s-1]$, and hence $m=n^r_{t}$ for $t=s-1$ and $r=m_{s-1}$. Since $m_{s-1} = k m_s  +q(c_s,d_s)$, $m_s\leq k-2$ and $q(c_s,d_s)\leq 2k-2$, we get that $r=m_{s-1}<k^2-1$, which ends the proof of the claim.
\end{proof}

\begin{remark}\label{rem:effectiveS} Note that the proof of Theorem \ref{prop:recurrence} is effective, in the sense that given a $k$-automatic sequence and words $u,v,w$ satisfying the conditions of the proposition, one may explicitly determine the set $S=\{n \geq 0\mid vu^nw \textrm{ is a factor of } x\}$ as a finite union of translates of geometric progressions and a finite set. In fact, even if $u, v, w$ fail to satisfy the assumptions, i.e.\ either $v$ is a suffix of some power of $u$ or $w$ is a prefix of some power of $u$, we may still determine $S$ as either an explicit finite set or as all of $\N$. Indeed, the  reasoning in Steps I and II in the proof of  Theorem \ref{prop:recurrence} allows us to assume that $u$, $v$ and $w$ are single letters. Fix $u$ and vary $v$ and $w$. If $v,w\neq u$, we already know how to find the corresponding set $S$. The remaining cases reduce to this one since every factor $vu^n w$ of $x$ is either a subfactor of some factor $\tilde{v}u^m\tilde{w}$ of $x$ for some letters $\tilde{v}, \tilde{w}$ with $\tilde{v},\tilde{w}\neq u$ or else arises as a factor of some prefix of $x$ of the form $u^m w$ or finally we have that $x$ is ultimately periodic with suffix $u^{\omega}$, in which case it is easy to find $S$. \end{remark}

\begin{corollary}\label{cor:commonfactors}  
Let $k,l\geq 2$ be multiplicatively independent integers and let $\mathcal{A}$ be an alphabet. Let $x$ be a $k$-automatic sequence over $\mathcal{A}$ and let $y$ be an $l$-automatic sequence over $\mathcal{A}$. Let $u,v,w$ be nonempty words over $\mathcal{A}$. Assume that $v$ is not a suffix of $u^n$ and $w$ is not a prefix of  $u^n$ for any integer $n$. Then the word $vu^nw$ is a common factor of $x$ and $y$ only  for finitely many $n$.
\end{corollary}
\begin{proof}
 Let $S_x=\{n\in\N\mid vu^nw \textrm{ is a factor of } x\}$ and $S_y=\{n\in\N\mid vu^nw \textrm{ is a factor of } y\}$. By  Theorem \ref{prop:recurrence}, $S_x$ is a finite union of sets of the form $\{ak^{mn}+b\mid n\geq 0\}$ for some $a,b\in\Q$ and $m\geq 0$. Similarly, $S_y$ is a finite union of sets of the form $\{al^{mn}+b\mid n\geq 0\}$ for some $a,b\in\Q$ and $m\geq 0$. In order to prove that the set $S_x\cap S_y$ is finite, it suffices to note that for any multiplicatively independent integers $k,l$ and rational numbers $a,b,c\in \Q$ with $a,b,c$ not all equal to zero  the exponential diophantine equation 
\[ak^n+bl^m=c\]
has only finitely many integer solutions $m,n\in\N$. This follows, e.g.\ from the finiteness of the number of solutions of $S$-unit equations due to Mahler \cite{Mahler1933} (see also \cite[Ch.\ 4]{book:EG} or \cite[p.\ 28]{Lang60} for a more general, but very convenient, statement).
\end{proof}

\subsection*{Proof of the main result}

We are now ready to prove Theorem \ref{thm:maincobham}. We begin with a lemma.

\begin{lemma}\label{lemmaformainCobham} Let $u,\tilde u, v$ be words over an alphabet $\mathcal{A}$, and let $n, m\geq 0$ be integers. Assume that $u$ and $\tilde u$ are primitive and that $\tilde u^m$ is a suffix of $u^{n} v$ with $m|\tilde u| \geq |v|+ |u|+|\tilde u|-\mathrm{gcd}(|u|,|\tilde u|)$. Then $|u|=|\tilde u|$, $u$ and $\tilde u$ are cyclic shifts of each other, and for any $q \geq 0$ 
we have $u^{n} v \tilde u^{q} = u^{n+q}  v$. 
\end{lemma}
\begin{proof} Let $\tilde v$ denote the word $\tilde u^m$ with the suffix $v$ removed. It follows from the Fine--Wilf theorem (see, e.g.\ \cite[Thm.\ 1.5.6]{AlloucheShallit-book}) applied to the (backwards infinite) periodic sequences  ${}^{\omega}u$ and ${}^{\omega}\tilde u\tilde v$  that ${}^{\omega}u = {}^{\omega}\tilde u\tilde v$, and hence  ${}^{\omega}uv = {}^{\omega}\tilde u$. Since $u$ and $\tilde u$ are primitive, $|u|=|\tilde u|$ and $u, \tilde u$ are cyclic shifts of each other.  Both the words $u^{n} v \tilde u^{q}$ and $u^{n+q} v$ are suffixes of ${}^{\omega}uv = {}^{\omega}\tilde u$ of the same length, and hence are equal. 
\end{proof}

\begin{proof}[Proof of Theorem \ref{thm:maincobham}]

We first prove that \eqref{thm:maincobhamii} implies \eqref{thm:maincobhami}. Write $U$ in the form \[U=\bigcup_{i=1}^{p} \mathcal{L}( {}^{\omega}v_i u_i w_i^{\omega}), \quad u_i,v_i,w_i \in \mathcal{A}^*, \quad p\geq 1.\]

Replacing $k$ and $l$ with their powers, we may assume that $k,l\geq p+2$. Choose two distinct symbols $\clubsuit$ and $\spadesuit$  not belonging to $\mathcal A$. We will construct a $k$-automatic sequence $x\in(\mathcal{A}\cup\{\clubsuit\})^{\omega}$ whose set of factors not containing $\clubsuit$ coincides with $U$.  Consider the sequences $(c^i_n)$ and the backwards infinite sequences $(d^i_n)$ defined as follows: \[c^i= \begin{cases} u_i w_i^{\omega}& \text{if } w_i\neq \epsilon,\\   u_i \clubsuit^{\omega}& \text{if } w_i = \epsilon;\end{cases}\qquad d^i= \begin{cases} {}^{\omega}v_i& \text{if }v_i\neq \epsilon,\\  {}^{\omega}\clubsuit &\text{if } v_i = \epsilon.\end{cases}\]  Since these sequences are ultimately periodic, they are both $k$- and $l$-automatic. Define the sequence $x\in (\mathcal{A}\cup\{\clubsuit\})^{\omega}$ by the formula 
\[
x_n=
\begin{cases} 
c^i_{n-i\cdot k^t} &\text{if }  i \cdot k^t\leq n < i\cdot k^t+k^{t-1},  t\geq 1, 1\leq i\leq p,\\
d^i_{n-i\cdot k^t} &\text{if }   i\cdot k^t-k^{t-1}\leq n < i\cdot k^t, t\geq 1, 1\leq i\leq p, \\
\clubsuit &\text{otherwise}.
\end{cases}
\]
(Recall that we always regard sequences in  $\mathcal{A}^{\omega}$ as indexed with $0,1,\ldots$, and sequences in ${}^{\omega}\!\mathcal{A}$ as indexed with $\ldots,-2,-1$.)
We claim that the sequence $x$ is $k$-automatic; this can be argued using either the characterisation of automaticity in terms of kernels or in terms of finite automata, and is elementary (albeit tedious). For a more precise argument, let $\varpi$ be the lowest common multiple of the lengths of those $w_i$ and $v_i$ that are nonempty, let $\varpi_0$ and $m_0$ be integers such that $\varpi_0>0$ and $k^{t+\varpi_0} \equiv k^{t} \pmod \varpi$ for $t\geq m_0$, and let $m_1 = \max|u_i|$. Then, for a given $n$ the value $x_n$ depends on the following data: a) the first two digits of $n$ in base $k$ (assuming $n\geq k$); b) the value of $n \bmod \varpi$; c) the value of $\lfloor \log_k n \rfloor \bmod \varpi_0$ (assuming $n\neq 0$); d) whether or not $n<\max(k^{m_0},k{m_1})$, and if so, what is $n$; e) whether or not $n$ is of the form $n=i\cdot k^t + j$ for some integers $i,j,t$ with $1\leq i \leq p$, $t\geq 1$ and $0\leq j<m_1$, and if so, what is the value of $j$ (note that the value of $j$ is uniquely determined for $n\geq k{m_1}$). Each of these pieces of information defines a $k$-automatic sequence, and thus $x$ itself  is $k$-automatic as a function of finitely many $k$-automatic sequences.

It is immediate that the set of factors of $x$ not containing the symbol $\clubsuit$ coincides with $U$. Replacing $k$ by $l$ and $\clubsuit$ by $\spadesuit$, we define an $l$-automatic sequence $y\in (\mathcal{A}\cup\{\spadesuit\})^{\omega}$  whose set of factors not containing $\spadesuit$  coincides with $U$. It follows that the set of common factors of $x$ and $y$ is exactly $U$. This ends the proof that  \eqref{thm:maincobhamii} implies \eqref{thm:maincobhami}. 

For the proof that \eqref{thm:maincobhami} implies \eqref{thm:maincobhamii}, let $x$ be a $k$-automatic sequence and let $y$ be an $l$-automatic sequence.  For simplicity, in the rest of the proof we will refer to common factors of $x$ and $y$  simply as common factors. It follows from Corollary \ref{cor:minimalfinite} that there are only finitely many primitive cyclic common factors (in fact, both $x$ and $y$ have only finitely many primitive cyclic factors). 
Let $\ell$ denote the maximal length of such a factor. 
We write common factors $t$  in the form
\begin{equation}\label{representation}
t=v_0 u_1^{n_1}v_1u_2^{n_2}\cdots v_{s-1}u_s^{n_s}v_{s}
\end{equation}   for some integer $s\geq 0$, integers $n_i \geq 0$ and words $u_i$, $v_i$ satisfying the following properties: \begin{enumerate} \item the words $u_i$ are primitive cyclic common factors, \item the words $v_i$ have length $|v_i|\leq \ell$, \item the integer $s$, called the \emph{rank}, is the smallest possible, \item given the choice of $s$, the sequence of integers $(n_1,\dots, n_s)$, called the \emph{sequence of exponents}, is lexicographically maximal. \end{enumerate} 

The fact that we allow $n_i$ to be zero guarantees the existence of such a representation for $t$, and we will always refer to it as the \emph{representation} of a common factor $t$. 

Note that if $t'$ is a prefix of $t$, then the rank of $t'$ is at most equal to the rank of $t$. We will prove that common factors have bounded rank. To this end, we first prove the following claim.

{\bf Claim 1}: For each $i\in \{2,\dots,s-1\}$ if $n_i > 0$,  then there exist \begin{enumerate} \item\label{thm:mainCobhamclaim1.1} a suffix $z$ of $v_0 u_1^{n_1} \cdots u_{i-1}^{n_{i-1}} v_{i-1}$ of length $|z| \leq 4\ell$ such that $z$ is not a suffix of $u_i^n$ for any integer $n\geq 0$; and \item\label{thm:mainCobhamclaim1.2}  a prefix $z'$ of $v_i u_{i+1}^{n_{i+1}} \cdots u_s^{n_s} v_s$ of length $|z'|\leq 4\ell$ such that $z'$ is not a prefix of $u_i^n$ for any integer $n\geq 0$.\end{enumerate}

\emph{Proof of Claim 1}: We only prove \eqref{thm:mainCobhamclaim1.1}, the proof of \eqref{thm:mainCobhamclaim1.2} being analogous. Write 
\[v_0 u_1^{n_1} \cdots u_{i-1}^{n_{i-1}} v_{i-1} = w' w,\]  
where $w,w' \in \mathcal{A}^*$, $w$ is a suffix of $u_i^n $ for some $n \geq 0$, and $w$ is chosen as long as possible. Let $m\geq 0$ be the largest integer such that $u_i^m$ is a suffix of $w$, and write $w=w'' u_i^m$. Note that $|w''|<\ell$.

Consider the following cases (which cover all possibilities): \begin{enumerate} \item If $w'=\epsilon$, then $t$ admits the representation $t=w'' u_i^{m+n_i} v_i  \cdots v_{s-1} u_s^{n_s}v_s$, which is of rank $s-i+1 <s$. This is a contradiction. 
\item If $|u_i^m| > |u_{i-1}^{n_{i-1}} v_{i-1}|$, we claim that $t$ also admits a representation of smaller rank. In fact, the word $w' w''$ is a prefix of $v_0 u_1^{n_1} \cdots u_{i-2}^{n_{i-2}} v_{i-2}$, and hence has a representation of rank at most $i -2$. Concatenating it with $u_i^{m+n_i}v_i\cdots v_{s-1}u_s^{n_s}v_{s}$, we obtain  a representation of $t$ of rank $\leq s-1$. This is a contradiction.
 \item If $|u_i^m| \leq |u_{i-1}^{n_{i-1}} v_{i-1}|$ and $|w| \geq 4\ell$, then $|u_i^m| > 3\ell$, and hence by  Lemma \ref{lemmaformainCobham} we may write $ u_{i-1}^{n_{i-1}} v_{i-1} u_i^{n_i} = u_{i-1}^r v_{i-1}$ with $r=n_{i-1}+n_i$. Replacing in the representation of $t$ the word $ u_{i-1}^{n_{i-1}} v_{i-1} u_i^{n_i} v_i $ by $u_{i-1}^r v_{i-1} u_i^{0}v_i$, we obtain a representation whose sequence of exponents is $(n_1, \dots, n_{i-2}, r, 0, n_{i+1}, \dots, n_s)$, and hence (since $n_i >0$) is lexicographically larger than  $(n_1, \dots, n_{i-2}, n_{i-1}, n_i,  n_{i+1}, \dots, n_s)$. This is a contradiction.
\item If $|w| <4\ell$ and $w' \neq \epsilon$, then  the suffix $z$ of $w'w$ of length $|w|+1$ satisfies the claim.\qed
\end{enumerate}

{\bf Claim 2}: There is a constant $C$ such that for any common factor $t$ the values of $n_2, \dots, n_{s-1}$ in any representation \eqref{representation} of $t$ are bounded  by $C$. 

\emph{Proof of Claim 2}:  Since there are only finitely many primitive cyclic common factors and finitely many words of length $\leq 4\ell$,  Corollary \ref{cor:commonfactors} and Claim 1 show that the values of $n_2, \dots, n_{s-1}$ are bounded by a constant independent of $t$. \qed

{\bf Claim 3}: The rank of common factors is bounded. 

\emph{Proof of Claim 3}:  Let $t$ be a common factor with representation  $t=v_0 u_1^{n_1}v_1u_2^{n_2}\cdots v_{s-1}u_s^{n_s}v_{s}$. We claim that any cyclic factor of $t$ (not necessarily primitive) can occur at positions intersecting at most three of the $u_i$'s.  Suppose this is not the case and write such a factor in the form $\tilde{u}^n$ for some integer $n\geq 0$ and primitive cyclic common factor $\tilde{u}$. Consider in the representation of $t$ the shortest factor $\tilde{w}$ consisting of a concatenation of $u_i$'s and $v_i$'s and containing $\tilde{u}^n$.  Replace the part of the representation of $t$ that is equal to $\tilde{w}$ by an expression of the form $v' \tilde{u}^n v''$ with $|v'|, |v''| \leq \ell$; if the words $v'$ and $v''$ are adjacent to some $v_i$, regard them as separated by $\tilde{u}^0=\epsilon$. In this manner, we obtain a representation of smaller rank; a contradiction. 

Now suppose that the rank of common factors is unbounded. Then the words 
\[w(t):=u_2^{n_2}v_2\cdots v_{s-2}u_{s-1}^{n_{s-1}}\]
can be arbitrarily long, and hence by compactness of $\mathcal{A}^{\omega}$ there exists a sequence $z$ with arbitrarily long prefixes of the form $w(t)$ for some common factors $t$. Let $X$ and $Y$ denote the orbit closures of $x$ and $y$, respectively. Then $z\in X \cap Y$, and hence by Corollary \ref{cor:cobpoints} it is ultimately periodic. However,  any cyclic factor of $w(t)$ occurs at positions  intersecting at most three of the $u_i$'s, and hence by Claim 2 it has length at most $(3C+4) \ell$, which is a contradiction. (In fact, a more careful reasoning gives the bound $(3C+2)\ell -2$.) \qed

{\bf Claim 4}: There exists a constant $C'$ such that any common factor $t$ can be written in the form $t=v' u^n v w^m v''$ for integers $n, m \geq 0$, primitive cyclic common factors $u,w$ and words $v, v', v''$ of length at most $C'$.

\emph{Proof of Claim 4}:  By Claim 3, the rank $s$ of a common factor $t=v_0 u_1^{n_1}v_1u_2^{n_2}\cdots v_{s-1}u_s^{n_s}v_{s}$ is bounded, and by Claim 2, $n_i\leq C$ for $i=2,\dots, s-1$. Since for all $i$, $|v_i|\leq \ell$ and $|u_i|\leq \ell$, the claim is satisfied with $v'=v_0$, $v''=v_s$, $u=u_1$, $w=u_s$, and $v=v_1u_2^{n_2}\cdots u_{s-1}^{n_{s-1}}v_{s-1}$.\qed

Due to Claim 4, in order to prove our result, it is sufficient to study common factors $t$ of the form $t=v' u^n v w^m v''$ for fixed words $u, w, v, v', v''$. Call a set $S$ of common factors special if it takes  the form \begin{align*} S&= \{t=v' u^n v w^m v''  \mid t \text{ is a common factor}, n,m \in \N\}\end{align*} for some words $u,w,v,v',v'' \in \mathcal{A}^*$. If we may further take $w=v''=\epsilon$, we call the set $S$ degenerate. By Claim 4, the set of all common factors is a finite union of special sets. We write $\mathcal{L}(S)$ for the set of all factors of words in $S$. We will prove that  for any special set $S$ the set $\mathcal{L}(S)$ is a finite union of sets of the form $\mathcal{L}({}^{\omega}\tilde{u} \tilde{v} \tilde{w}^{\omega})$ with $\tilde{u}, \tilde{v}, \tilde{w} \in \mathcal{A}^*$. This will conclude the proof of the theorem.

We first prove the claim for a degenerate special set \[S= \{t=v' u^n v  \mid t \text{ is a common factor}, n\in \N\}.\] If the set $S$ is finite, then it is certainly of the desired form. By Corollary \ref{cor:commonfactors} this is the case if  $v'$ is not a suffix of any power of $u$ and $v$ is not a prefix of any power of $u$. If on the other hand $v'$ is a suffix of some power of $u$, and $S$ is infinite, then $\mathcal{L}(S)$ is equal to $\mathcal{L}({}^{\omega}u v)$. A similar reasoning proves the claim if  $v$ is a prefix of some power of $u$.

Consider now the case of a general special set \[S= \{t=v' u^n v w^m v'' \mid t \text{ is a common factor}, n,m \in \N\}.\] If $S$ does not contain factors of the form $t=v' u^n v w^m v''$ for arbitrarily large values of both $n$ and $m$, then $S$ can be rewritten as a finite union of degenerate special sets, and the claim follows. Suppose that $S$ contains factors $t$ corresponding to arbitrarily large values of both $n$ and $m$. If either $v w^m$ is a prefix of some power of $u$ for arbitrarily large $m$ or $u^n v$ is a suffix of some power of $w$ for arbitrarily large $n$, then we may again rewrite $S$ as a finite union of degenerate special sets. Finally, if neither is $v w^m$ a prefix of some power of $u$ for sufficiently large $m$ nor is $u^n v$  a suffix of some power of $w$ for sufficiently large $n$, then we  conclude from Corollary  \ref{cor:commonfactors} that $v'$ is a suffix of some power of $u$ and $v''$ is a prefix of some power of $w$. In this case the set $\mathcal{L}(S)$ is equal to $\mathcal{L}({}^{\omega}u v w^{\omega})$, which finishes the proof.
\end{proof}

It is an interesting question whether Theorem \ref{thm:maincobham} can be made effective.

\begin{question} Is there an algorithm which, given a $k$-automatic sequence $x$ and an $l$-automatic sequence $y$, produces words $u_i$, $v_i$, $w_i$, $1\leq i \leq p$, such that the set $U$ of common factors of $x$ and $y$ is equal to $U=\bigcup_{i=1}^p \mathcal{L}({}^{\omega}v_i u_i w_i^{\omega})$?
\end{question}

\begin{remark} The only place in the proof of Theorem \ref{thm:maincobham} where it is not clear if the proof is effective is the bound on the rank of common factors (Claim 3), which uses a compactness argument. Let us briefly comment on how to make other parts of the proof effective. 

First of all, we can  determine all primitive cyclic factors of an automatic sequence $x$. In fact, write $x$ as the image of a fixed point of a substitution under a coding. Replace the substitution by an idempotent one using Lemma \ref{lem:propofsubs}, which is effective. Proposition \ref{prop:minimalsubsystems} together with Lemma \ref{lem:factor} describe all minimal subsystems of $\overline{\Orb(x)}$ as closures of orbits of explicitly given automatic sequences. To find the cyclic factors of $x$, we need to determine which of these automatic sequences are periodic, for which a decision procedure was given by Honkala \cite{Honkala} (see also \cite{ARS09} for a simpler approach).
Another crucial  ingredient of the proof is  Corollary \ref{cor:commonfactors}, which uses the $S$-unit equation. Here, solutions can be effectively bounded using Baker's method (for a comprehensive discussion, see \cite{book:EG}). In particular, the constant $C$ in the proof can be effectively computed. Finally, given words $v, v', v'', u, w$ we may effectively determine all common factors of $x$ and $y$ of the form $t=v' u^n v w^m v''$ using Remark  \ref{rem:effectiveS} and an effective version of  Corollary \ref{cor:commonfactors}. Thus, in order to make the proof fully effective, we need to find a computable bound on the rank of common factors or---equivalently---the constant $C'$. Unfortunately, we do not know how to do this.
\end{remark}

\bibliographystyle{amsplain}
\bibliography{bibliography}
\end{document}